\documentclass[10pt]{amsart}
\usepackage[hyperfootnotes=false]{hyperref}
\usepackage{enumerate}
\usepackage{tikz}
\usepackage{tikz-cd}
\usepackage{hyperref}
\usepackage{amsthm,amsfonts,amsmath,amscd,amssymb}
\usepackage{xcolor}
\usepackage{mathrsfs}
\usepackage{graphics,graphicx}
\usepackage{caption}
\let\oldmarginpar\marginpar
\renewcommand\marginpar[1]{\-\oldmarginpar[\raggedleft\footnotesize #1]%
{\raggedright\footnotesize #1}}
\theoremstyle{plain}
\newtheorem{thm}{Theorem}[section]
\newtheorem{lemma}[thm]{Lemma}

\newtheorem{prop}[thm]{Proposition}
\newtheorem{cor}[thm]{Corollary}

\theoremstyle{definition}

\newtheorem{definition}[thm]{Definition}
\newtheorem{remark}[thm]{Remark}

\theoremstyle{remark}

\numberwithin{equation}{section}

\newcommand{\Hom}{\text{Hom}}

\newcommand{\F}{\mathbb{F}}

\newcommand{\N}{\mathbb{N}}
\newcommand{\Z}{\mathbb{Z}}

\newcommand{\R}{\mathbb{R}}
\newcommand{\C}{\mathbb{C}}

\newcommand{\SB}{\mathcal{B}}

\newcommand{\SL}{\mathcal{L}}
\newcommand{\PP}{\mathbb{P}}

\newcommand{\A}{\mathcal{A}}
\newcommand{\La}{\Lambda}

\newcommand{\p}{\varphi}
\renewcommand{\a}{\alpha}

\newcommand{\e}{\varepsilon}
\newcommand{\dd}{\partial}

\newcommand{\op}{\operatorname}

\newcommand{\sse}{\subseteq}
\newcommand{\lr}{\longrightarrow}

\newcommand{\x}{\times}
\newcommand{\sm}{\setminus}

\newcommand{\Aut}{\operatorname{Aut}}

\newcommand{\wt}{\widetilde}
\newcommand{\wh}{\widehat}

\newcommand{\Int}{\operatorname{Int}}

\newcounter{daggerfootnote}
\newcommand*{\daggerfootnote}[1]{%
    \setcounter{daggerfootnote}{\value{footnote}}%
    \renewcommand*{\thefootnote}{\fnsymbol{footnote}}%
    \footnote[2]{#1}%
    \setcounter{footnote}{\value{daggerfootnote}}%
    \renewcommand*{\thefootnote}{\arabic{footnote}}%
    }

\begin{document}
\begin{abstract}
In this article we associate a combinatorial differential graded algebra to a cubic planar graph $G$. This algebra is defined combinatorially by counting binary sequences, which we introduce, and several explicit computations are provided. In addition, in the appendix by K.~Sackel the $\F_q$--rational points of its graded augmentation variety are shown to coincide with $(q+1)$-colorings of the dual graph.
\end{abstract}

\title{Differential algebra of cubic planar graphs}

\author{Roger Casals}
\address{Massachusetts Institute of Technology, Department of Mathematics, 77 Massachusetts Avenue Cambridge, MA 02139, USA}
\email{casals@mit.edu}

\author{Emmy Murphy}
\address{Northwestern University, Department of Mathematics, 2033 Sheridan Road Evanston, IL 60208, USA}
\email{e\_murphy@math.northwestern.edu}

\maketitle
\section{Introduction}\label{sec:intro}
This article defines new algebraic structures associated to cubic graphs. The inspiration for our construction comes from symplectic field theory \cite{EES,EGH} and the theory of constructible sheaves \cite{NZ,STZ}. To every planar cubic graph, we assign to it a differential graded algebra, which describes a number of combinatorial features of a graph. In particular, the augmentation variety of this algebra recovers the chromatic data of the dual graph, established in Appendix \ref{app:kevin}, written by K.~Sackel. In terms of graph Legendrians, this can be interpreted as the algebraic part of the conjectural correspondence between the augmentation variety of the Legendrian contact homology algebra and the moduli space of rank--1 constructibles sheaves \cite{NRSSZ,TZ}, which also constitute part of the developing program for mirror symmetry \cite{AV2,AENV}.

The combinatorics presented in this work lie at the intersection of three fields: the contact topology of Legendrian surfaces \cite{Ar,Be}, the combinatorics of cubic graphs \cite{Di,St} and the enumerative geometry of open Gromov-Witten invariants \cite{AV1,FL}. In addition, the structures we introduce have some interesting connections to spectral networks \cite{GMN1,GMN2}: the binary sequences arising in our construction distinguish a special type of framed 2d-4d BPS states of the supersymmetric $N=2$ 4d-theories of class $S$ associated to the Lie algebra $\mathfrak{g}=su(2)$. These connections will be explored in more depth in a later work.

For now, this article defines and explores this algebraic structure from a purely combinatorial perspective. Given a planar cubic graph $G$ with $2g+2$ vertices and a ground field $\F$, we consider the base ring $\Lambda_G = \F[e^{\pm1}_1,\ldots,e^{\pm1}_{3g+3}]$ of Laurent polynomials in the set of edges, and the unital $\Lambda_G$--algebra generated by the faces of $G$ and three additional generators:
$$\wt\A_G=\langle x,y,z,f_j: 1\leq j\leq g+2\rangle.$$

The goal is to endow $\wt\A_G$ with the structure of a dg--algebra which contains interesting information about the graph, and in particular show that it contains the chromatic polynomial. Here is an example: consider the modified 4-prism graph in the left of Figure \ref{fig:blowcube}. We equip the graph $G$ with blue and red paths as in the right of Figure \ref{fig:blowcube}, together we denote the choice of these decorations by $\Gamma$.

\begin{figure}[h!]
\centering
  \includegraphics[scale=0.75]{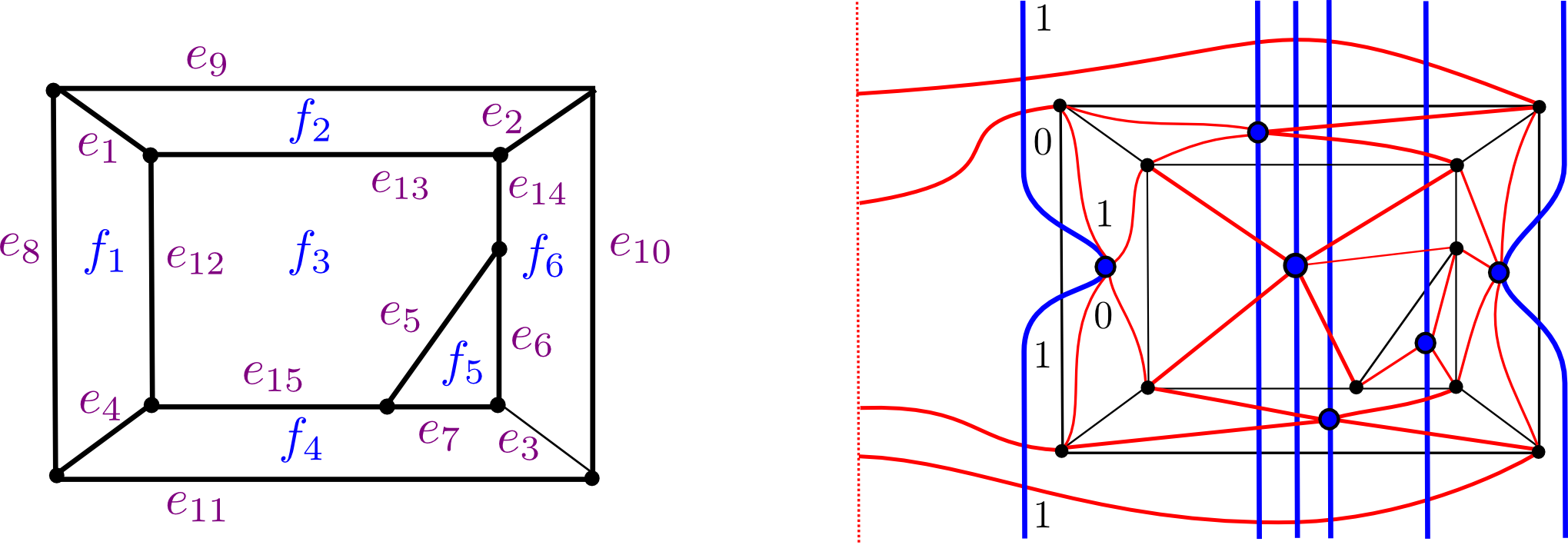}
  \caption{Planar cubic graph $G$ (left) and its decoration (right).}\label{fig:blowcube}
\end{figure}

Then the count of binary sequences along the blue paths interacting with the red paths allows us to define a differential operator $\wt\dd_\Gamma$ on $\wt\A_G$, whereas the red paths shall dictate the differential on the faces $F=\{f_1,\ldots,f_6\}$. For instance, in the case of Figure \ref{fig:blowcube} above the differential of $x$ contains the information of the binary sequence along the leftmost vertical blue path. The rules to count binary sequences are explained in Definition \ref{def:total A} in Section \ref{sec:variants}.

As the example shows, the graph $G$ is first endowed with several decorations $\Gamma$ before the dg--structure can be defined. These decorations are not canonical but necessary in order to be able for binary sequences to interact with the graph $G$ and endow $\wt\A_G$ with a differential structure $\wt\dd_\Gamma$. From the combinatorial perspective, it seems quite outstanding that, as we show, the resulting algebraic structure is independent of these choices. The central result, which includes proving such independence, is the definition of this differential operator $\wt\dd_\Gamma:\wt\A_G\lr\wt\A_G$, and can be, in its bare form, stated in the following theorem:

\begin{thm}\label{thm:main}
The count of binary sequences on the decorated graph $(G,\Gamma)$ endows $\wt\A_G$ with the structure of a differential graded algebra $(\wt\A_G,\wt\dd_\Gamma)$. The dg--isomorphism type of the dg-algebra $(\wt\A_G,\wt\dd_\Gamma)$ only depends on $G$.
\end{thm}

Theorem \ref{thm:main} states that the dg-algebra isomorphism type of $(\wt\A_G,\wt\dd_G)$ is independent of all the involved choices $\Gamma$, which we prove in Section \ref{sec:invar}, and that the operator $\wt\dd_\Gamma$ is indeed a differential, i.e.~$\wt\dd_\Gamma^2=0$, which is the content of Section \ref{sec:dsquared}.

In addition, from this dg-algebra $(\wt\A_G,\wt\dd_G)$, which shall be called the dg-algebra of binary sequences, it is possible to extract different algebraic structures, which are explained in Section \ref{sec:real DGA}. For instance, there is a canonical action of $\Z$ on $(\wt\A_G,\wt\dd_\Gamma)$ acting by dg--isomorphisms and we introduce the dg-algebra $\A_G \sse \wt\A_G$ of fixed elements of this action, which is equipped with $\dd_\Gamma = \wt\dd_\Gamma|_{\A_G}$. This invariant dg--algebra $(\A_G, \dd_\Gamma)$ is the algebraic structure that leads to the combinatorics of graphs colorings:

\begin{thm}[K. Sackel, Appendix \ref{app:kevin}]\label{thm:main kevin}
Let $G$ be a planar cubic graph and $\chi_{G^*}$ the chromatic polynomial of its dual. Then
$$\chi_{G^*}(q+1)=(q^3-q)\cdot|\Hom_{\mbox{\tiny dga}}((\A_G, \dd_\Gamma),(\F_q,0))|.$$

In addition, let $X_G=\Hom_{\mbox{\tiny dga}}((\A_G, \dd_\Gamma),(\F,0))$ be the augmentation variety of $(\A, \dd_\Gamma)$ and $\chi_{G^*}(\F\mathbb{P}^1)$ the projective chromatic $\F$-variety of $G^*$. Then there is an isomorphism of algebraic varieties
$$X_G \cong \chi_{G^*}(\F\mathbb{P}^1)/\op{PSL}(2, \F),$$
where $\op{PSL}(2, \F)$ acts diagonally on $(\F\mathbb{P}^1)^{g+3}$.
\end{thm}

From the algebraic structure of $(\wt\A_G, \wt\dd_\Gamma)$ we can also consider the set of its quotient algebras. Given any subset of edges $\mathcal{E}$, we can define a quotient algebra of $(\wt\A_G, \wt\dd_\Gamma)$ by declaring the edges $e \in \mathcal{E}$ to be equal to $1$, and the dg-isomorphism type of the resulting algebra strongly depends in general on the choice of the subset $\mathcal{E}$. However, suppose we choose a tree $T\sse G$ which spans the set of vertices $V$ except for one, and let $\mathcal{E}=E\setminus T$ be the set of all edges which are not contained in $T$. Then the quotient algebra $(\A_T,\dd_T)$ associated to $E\setminus T$ is dg--isomorphic to the dg-algebra $(\A_G, \dd_\Gamma)$.

This is remarkable for two reasons: its dg--isomorphism type does not depend on the choice of tree $T$, and it gives an effective method for calculating $(\A_G, \dd_\Gamma)$ without having to consider a group action. From a geometric perspective, this corresponds to a gauge-fixing condition for a slice in the GIT quotient of $\op{Spec}(\wt\A_G, \wt\dd_\Gamma)$, which is the algebraic mirror of the Legendrian defined by $G$.\\

{\bf Acknowledgements}: We are grateful to Maxime Gabella, Richard Stanley, David Treumann and Eric Zaslow for useful conversations. R.~Casals is supported by the NSF grant DMS-1608018 and a BBVA Research Fellowship and E.~Murphy is partially supported by NSF grant DMS-1510305 and a Sloan Research Fellowship. E.~Murphy would like to thanks the Radcliffe Institute of Advanced Studies where part of the article was written while in residence.\hfill$\Box$

\section{The dg-algebra of binary sequences}\label{sec:variants}

Let us start with the basic definitions in homological algebra \cite[Chapter V.3]{GelMan}, where $\F$ is a field.

\begin{definition}\label{def:DGA}
A {\bf differential graded algebra} $(\A, \dd)$, or {\bf dg-algebra}, is a pair given by a $\Z$-graded $\F$-algebra $\A = \bigoplus_{i \in \Z} \A_i$ and a chain differential $\dd:\A\lr\A$, i.e. an $\F$--linear map which satisfies
\begin{itemize}
 \item[1.] $\dd^2=0$ and $\dd(\A_i)\sse\A_{i-1}$, $\forall i\in\Z$.
 \item[2.] $\dd(x\cdot y)=\dd(x)\cdot y+(-1)^{i}x\cdot\dd(y) \text{ for } x \in \A_i$.
\end{itemize}

A map $\p: (\A,\dd_A) \lr (\mathcal{B},\dd_B)$ between two dg-algebra is a {\bf dg-algebra map} if it is a unital ring map such that $\p \circ \dd_\A = \dd_{\mathcal{B}} \circ \p$ and $\p(\A_i) \sse \mathcal{B}_i$ $\forall i\in\Z$. A dg-algebra map which is also a bijection is called a {\bf dg-algebra isomorphism}.\hfill$\Box$
\end{definition}

\begin{remark}\label{rmk:iso vs chain iso}
The equivalence of dg-algebras given by a dg-algebra isomorphism is exceedingly strong in geometric contexts \cite{EES,EGH} and instead one typically works in the category of dg-algebras up to stable tame isomorphisms, chain isomorphisms, or quasi-isomorphisms. In our case, the dg-algebra in Theorem \ref{thm:main} is invariant up to dg-algebra isomorphisms, which is the strongest of these notions. This proves to be advantageous, as it allows us to construct more computable invariants than its homology, such as the characteristic algebra \cite[Section 3]{Ng}.\hfill$\Box$
\end{remark}

Let $G$ be a cubic planar graph equipped with a fixed embedding $G \sse (-1,1) \x (-1,1)$, and denote by $V, E$ and $F$ its sets of vertices, edges, and faces respectively; consider also the unique integer $g\in\N$ satisfying $|V| = 2g+2$, $|E| = 3g+3$ and $|F| = g+2$.

\begin{definition}\label{def:garden}
Let $G \sse (-1,1) \x (-1,1)$ be a planar cubic graph.
\begin{itemize}
\item[-] An {\bf orientation} of $G$ is a chosen orientation for each edge $e \in E$.\\

\item[-] A {\bf centering} of $G$ is a choice, for each face $f \in F$, a point $c_f$ in the interior of $f$, which is called the {\bf center} of $f$.\\

\item[-] A {\bf web} of $G$ is a collection of embedded closed arcs $W=\{\tau_f(v)\}$, indexed by $f\in F$ and $v\in f$, such that each $\tau_f(v)$ connects $v$ to $c_f$, $\Int(\tau_f(v))\sse \Int(f)$, and $\tau_f(v)$ are mutually disjoint from each other for all $f\in F$.\\

\noindent The set $W$ can be enlarged to $\widehat W$ by adding a collection of arcs $\{\tau_0(v)\}$, one for each vertex $v$ adjacent to the exterior of $G$, which connects $v$ to $\{-1\} \x [-1, 1] $. These $\tau_0(v)$ must be disjoint from $G$ away from $v$, and mutually disjoint from each other for all $f\in F$. The arcs in $\widehat W$ are called the {\bf threads} of $G$, and we refer to the threads in $\widehat W \sm W$ as the {\bf threads at infinity}.\\

\item[-] A {\bf rake} of $G$ is a choice, for each face $f \in F$, of a path $\gamma_f$ with endpoints respectively contained in $[-1, 1] \x \{-1\}$ and $[-1, 1] \x \{1\}$ which passes through $c_f$. The path $\gamma_f$ is called the {\bf tine} of $f$, and we require that all tines are mutually disjoint, disjoint from $V$, and all tines intersect all edges and all threads of $G$ transversely.\\
\end{itemize}
A {\bf garden} $\Gamma$ for $G$ consists of a choice of the above four decorations: a centering $\{c_f\}_{f\in F}$, an enlarged web $\widehat W$, a rake $\{\gamma_f\}_{f\in F}$, and an orientation.
\end{definition}

\begin{remark}
The pairs $(G,\Gamma)$ consisting of a graph $G$ decorated with a garden $\Gamma$ are to be considered up to smooth isotopy respecting the combinatorial structure of the graph and the incidences between the elements of the garden and the graph itself. In addition, we have adopted the color convention where the graph is depicted in black and given a garden $\Gamma$, the tines are depicted in the blue, and threads in red.\hfill$\Box$
\end{remark}

Given a garden and a tine $\gamma_f$, define $E(\gamma_f) \sse \gamma_f$ to be the set of all points where $\gamma_f$ intersects an edge and consider the set of points $\alpha(\gamma_f) \sse \gamma_f$ where $\gamma_f$ intersects the interior of a thread. Note that by definition, $c_f \notin \alpha(\gamma_f)$. The following definition contains the central combinatorial object of this article.

\begin{definition}\label{def:binary}
Let $G$ be a graph equipped with a garden $(\widehat W,\{c_f,\gamma_f\}_{f\in F})$ and orient each tine $\gamma_f$ from bottom to top. A {\bf binary sequence along} $\gamma_f$ is a lower semicontinuous map $B:\gamma_f \lr \{0,1\}$, which is constant outside of the set $E(\gamma_f) \cup \alpha(\gamma_f) \cup \{c_f\}$, and conforms to the following properties:
\begin{itemize}
\item[-] At all points of $E(\gamma_f)$ the value of $B$ must switch.
\item[-] At each $c_f$, the value of $B$ must switch from $0$ to $1$.
\item[-] At points $\alpha(\gamma_f)$ the value of $B$ may either switch from $0$ to $1$, or remain constant.
\end{itemize}
The values of $B$ at the respective endpoints of $\gamma_f$ are called the initial and final values of $B$.

Given a binary sequence $B$ along $\gamma_f$, we denote by $\a(B) \sse \alpha(\gamma_f)$ the set of points where $\gamma_f$ intersects the interior of a thread and in addition $B$ switches its value.\hfill$\Box$
\end{definition}

Let $\Lambda_G := \F[e_1, e_1^{-1}, \ldots, e_{3g+3}, e_{3g+3}^{-1}]$ be the commutative algebra of Laurent polynomials in the edge set $E$ and define $\wt \A_G$ to be the $\Lambda_G$--algebra freely generated by the faces of $G$ and three additional generators $x$, $y$ and $z$, i.e.:

$$\wt\A_G=\langle x,y,z,f_1,\ldots , f_{g+2}: f_j \in F\rangle_{\Lambda_G}.$$

Thus the elements of $\wt\A_G$ are finite sums of elemtents of the form $P(e)w$, where $P(e)$ is a Laurent polynomial in the edge set and $w$ is any word in the alphabet $\langle f_1, \ldots, f_{g+2}, x, y, z\rangle$; multiplication is given by formal concatenation and $\Lambda_G$ is the center of $\wt\A_G$. Endow the algebra $\wt\A_G$ with the grading induced by the following grading on the generators:
$$|x|=|y|=|z|=2,\quad |f_i|=1,\quad 1\leq i\leq g+2,$$
and $|e|=0$ for any edge $e\in E$.\\

We will now use binary sequences to define a chain differential $\wt\dd_G:\wt\A_G \lr \wt\A_G$, which requires the following definitions. Let $B$ be a binary sequence along $\gamma_f$. Given $q \in E(\gamma_f)$ an intersection between $\gamma_f$ and an edge $e\in E$ define
$$H(B, q) := \pm e^{\pm 1},$$
according to the rules in Figure \ref{fig:tineedgetable}.

\begin{figure}[h!]
\centering
\includegraphics[scale=0.6]{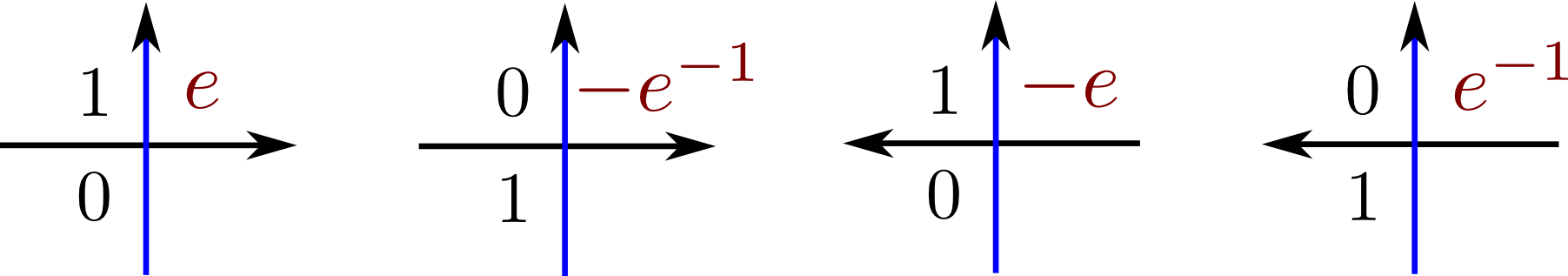}
\caption{The contributions $H(B,q)$ for an intersection points $q$ between a tine $\gamma_f$, depicted vertical and in blue, and an edge $e$, in black and horizontal.}
\label{fig:tineedgetable}
\end{figure}

The threads of the web $\widehat W$ contribute as follows. Let $r_v \in \{0,1,2,3\}$ to be the number of edges adjacent to a vertex $v\in V$ which are oriented outward. If $\tau$ is a thread we define
$$H(\tau) := (-1)^{r_v} e_n e_m e_k^{-1},$$
where $v$ is the endpoint of $\tau$, $e_n$ and $e_m$ are the edges containing $v$ which are adjacent to $\tau$, and $e_k$ is the edge containing $v$ which is opposite of $\tau$. 

\begin{figure}[h!]
\centering
\includegraphics[scale=0.75]{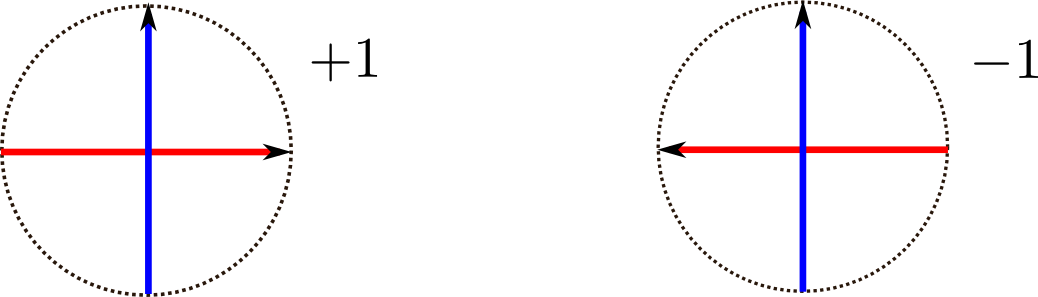}
\caption{The contributions $H(B,q)$ for an intersection points $q$ between a tine $\gamma_f$, depicted vertical and in blue, and a thread, in red and horizontal.}
\label{fig:tinethread}
\end{figure}

Finally, orient all threads $\tau$ towards the endpoint of $\tau$ contained in the vertex set $V$, and at every point $p \in \alpha(\gamma_f)$ where $\gamma_f$ intersects a thread $\tau$ there is a sign $\sigma(p)$ according to the oriented intersection as depicted in Figure \ref{fig:tinethread}. We define the $\Lambda_G$ contribution of $B$ as
$$H(B) := \prod_{q \in E(\gamma_f)} H(B,q)\cdot\prod_{p \in \alpha(B)} \sigma(p) H(\tau_p),$$
where $\tau_p$ denotes the thread containing $p\in\alpha(\gamma_f)$. The quantity $H(B)\in \La_G$ is a combinatorial index which the differential $\wt \dd$ counts.

\begin{definition}\label{def:total A}
Let $G$ be a graph equipped with a garden $\Gamma=(\widehat W,\{c_f,\gamma_f\}_{f\in F})$. Let $\mathcal{B}_{a,b}(f)$ denote the set of all binary sequences along $\gamma_f$ with initial value $a\in \{0,1\}$ and final value $b \in \{0,1\}$.

The {\bf differential of the dg--algebra of binary sequences} $(G,\Gamma)$ is the map
$$\wt\dd_\Gamma: \wt\A_G \to \wt\A_G$$
defined on generators as
$$\wt\dd_\Gamma f_j = \sum_{v \in f_j} H(\tau_f(v)),\qquad\forall f_j \in F,$$
$$\wt\dd_\Gamma x = \sum_{j=1}^{g+2} \left(\sum_{B \in \mathcal{B}_{1,1}(f_j)}H(B)\right)f_j,$$ $$\wt\dd_\Gamma y = \sum_{j=1}^{g+2} \left(\sum_{B \in \mathcal{B}_{1,0}(f_j)}H(B)\right)f_j,$$ $$\wt\dd_\Gamma z = \sum_{j=1}^{g+2} \left(\sum_{B \in \mathcal{B}_{0,0}(f_j)}H(B)\right)f_j.$$
The operator $\wt\dd_\Gamma$ extends by Leibniz's rule to $\wt\A_G$ with $\wt\dd_\Gamma e = 0$ for all $e \in E$.\hfill$\Box$
\end{definition}

The main result of this article is to show that Definition \ref{def:total A} yields the structure of a dg-algebra and such structure is independent of all the choices made to define it up to dg-isomorphism. In precise terms, we will prove the following theorem:

\begin{thm}\label{thm:total A}
Let $G$ be a planar cubic graph equipped with a garden $\Gamma=(\widehat W,\{c_f,\gamma_f\}_{f\in F})$, then $(\wt \A_G, \wt \dd_\Gamma)$ is a dg--algebra. In addition, the dg--algebra isomorphism type of $(\wt \A_G, \wt \dd_\Gamma)$ does not depend on the choice of garden $\Gamma$.
\end{thm}

The {\bf dg--algebra of binary sequences} of $G$ is the isomorphism type of $(\wt\A_G,\wt\dd_G)=(\wt\A_G,\wt\dd_\Gamma)$ where $\wt\dd_G=\wt\dd_\Gamma$ is defined for an arbitrary choice of garden $\Gamma$. Sections \ref{sec:invar} and \ref{sec:dsquared} are devoted to the proof of Theorem \ref{thm:total A}, which is the correctly refined statement of Theorem \ref{thm:main}.

\section{Independence of the choice of garden} \label{sec:invar}

Let $G$ be a planar cubic graph and $\Gamma_0,\Gamma_1$ two gardens, this goal of this section is to prove that there is a graded algebra isomorphism $\p\in\Aut(\wt\A_G)$ such that $\p\circ\wt\dd_{\Gamma_0}=\wt\dd_{\Gamma_1}\circ\p$. The choice of the rake is the most delicate matter, which is dealt in the following proposition.

\begin{prop}\label{prop:rake indep}
Let $G$ be a planar cubic graph equipped with two gardens $\Gamma_0,\Gamma_1$ which differ only at their respective rakes $R_0=\{\gamma^0_f\}$ and $R_1=\{\gamma^1_f\}$. Then there exists a graded algebra isomorphism $\p\in\Aut(\wt\A_G)$ such that $\p\circ\wt\dd_{\Gamma_0}=\wt\dd_{\Gamma_1}\circ\p$. In addition, the isomorphism can be chosen to fix $\Lambda_G \sse \wt\A_G$.
\end{prop}

In order to prove this, we shall use the following two lemmas.

\begin{lemma}\label{lem:rake moves}
Two rakes $R_0=\{\gamma^0_f\}\sse\Gamma_0$ and $R_1=\{\gamma^1_f\}\sse\Gamma_1$ differ by a finite sequence of the five moves depicted in Figures \ref{fig:Tangencies}, \ref{fig:VertexCrossing}, \ref{fig:TineSwitch} and \ref{fig:TineSwitch2} and a smooth plane isotopy.\hfill$\Box$
\begin{figure}[h!]
  \centering
  \includegraphics[scale=0.75]{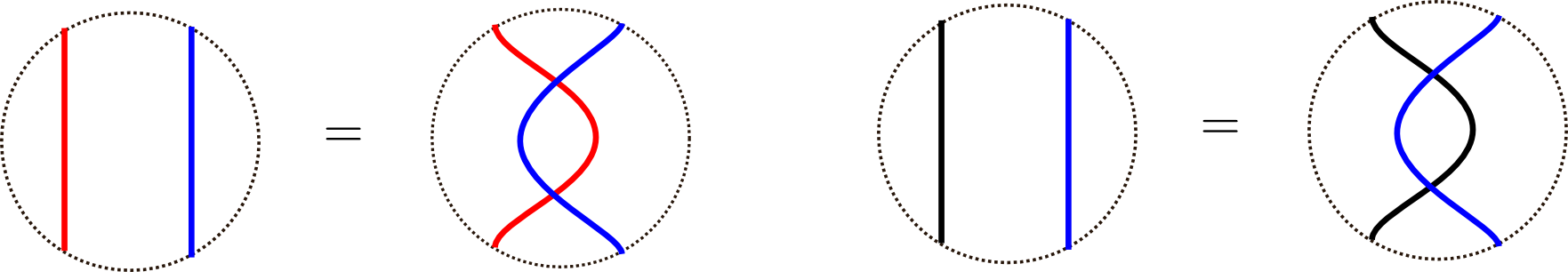}
  \caption{Moves I and II in Lemma \ref{lem:rake moves}: thread (red) and tine (blue) tangency on the left, and edge (black) and tine (blue) tangency on the right. These moves are referred to as tangencies since a unique point of tangency occurs when the strands on the right hand side of the equalities are horizontally separated to the left hand side configurations.}\label{fig:Tangencies}
  \vspace{0.3cm}
  
  \includegraphics[scale=0.85]{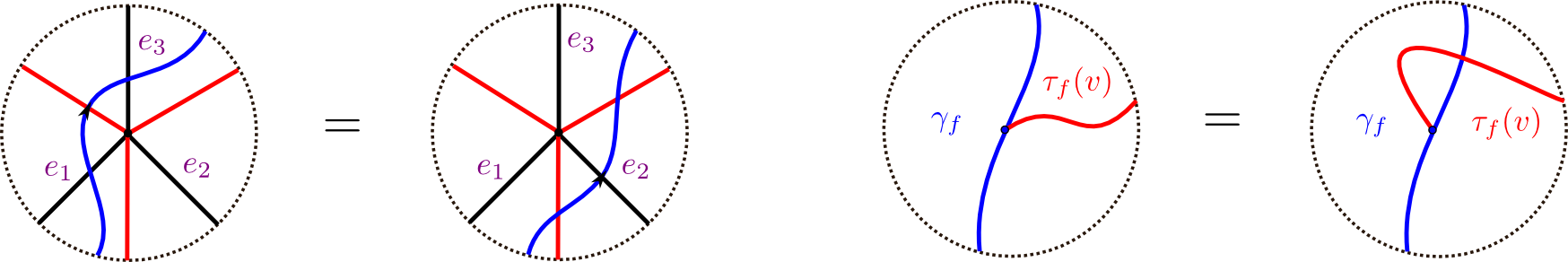}
  \caption{Move III and IV: a tine (blue) crossing a vertex, with threads (red) and edges (black), and a rotation of a thread (red) along the center of a face (blue dot).}\label{fig:VertexCrossing}
  \vspace{0.3cm}
  
  \includegraphics[scale=0.55]{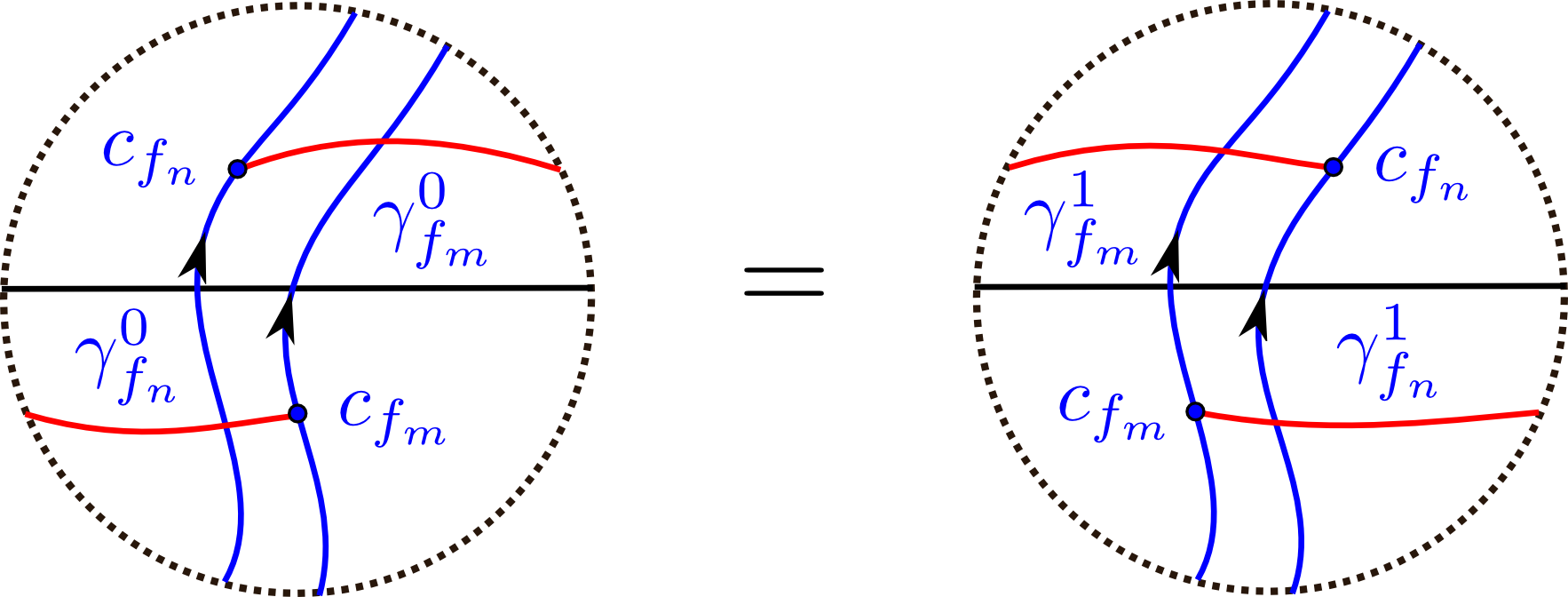}
  \caption{Move V, Part I: tine switch between two adjacent tines $\gamma_{f_n},\gamma_{f_m}$ in the rake with the centers separated by a unique edge.}\label{fig:TineSwitch}
  \vspace{0.3cm}
  
  \includegraphics[scale=0.55]{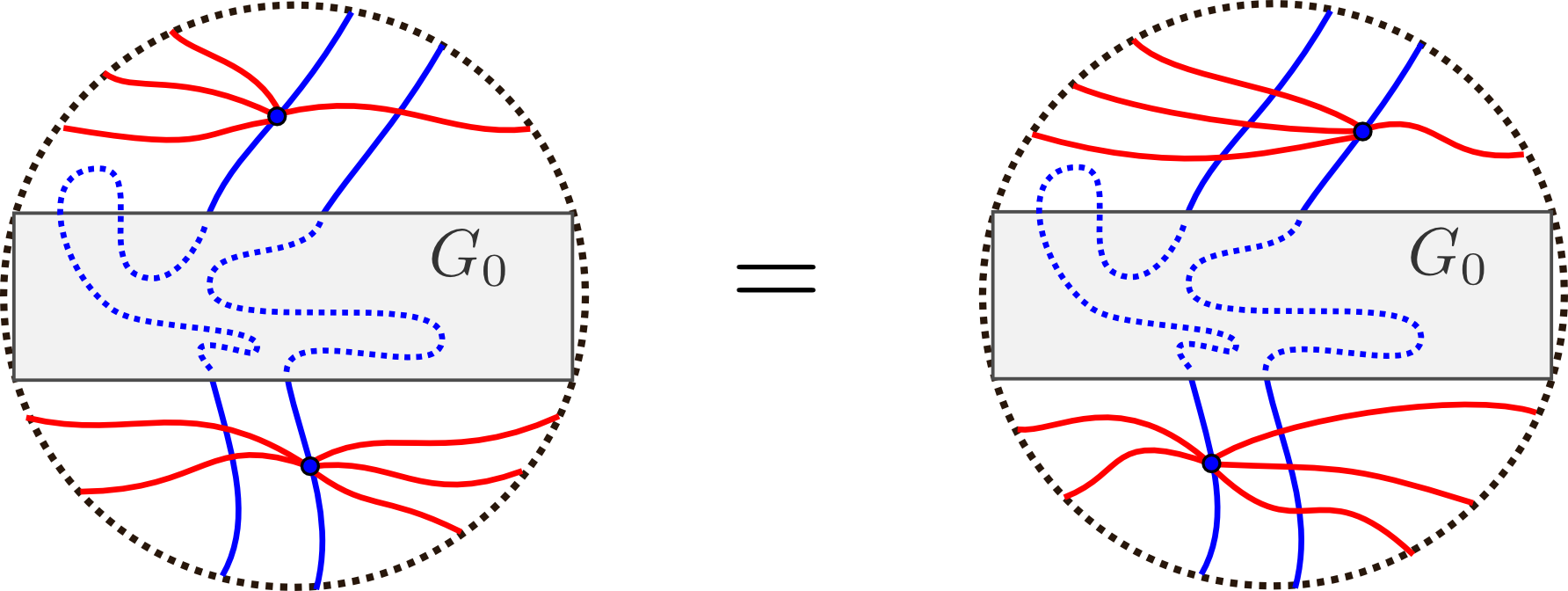}
  \caption{Move V, Part II: tine switch in the general setting with a piece $G_0$ of the graph in between; in this depiction the subgraph $G_0$ might be arbitrarily complicated.}\label{fig:TineSwitch2}
\end{figure}
\end{lemma}

The proof of Lemma \ref{lem:rake moves} is strictly combinatorial and it is left as an exercise for the reader. In contrast, the following lemma, for which we provide a detailed proof, features the operator $\wt\dd_\Gamma$.

\begin{lemma}\label{lem:easy rake moves}
The operator $\wt\dd_\Gamma$ is invariant under moves \emph{I, II, III} and \emph{IV}.
\end{lemma}

\begin{proof}
For Move I in Figure \ref{fig:Tangencies}, where a tine and a thread interact, the count of binary sequences is as follows. On the local situation depicted on its left side we have $|\SB_{1,1}|=|\SB_{0,0}|=1$ given by the constant sequences, whereas $\SB_{1,0}=\SB_{0,1}=\emptyset$; on the right hand side the constant sequences persist and still $|\SB_{1,1}|=|\SB_{0,0}|=1$, and it is apparent that $\SB_{1,0} = \emptyset$. However, $|\SB_{0,1}|=2$ depending on which point of the thread we switch: in one of these sequences the intersection index $\sigma$ is positive and the other is negative. Hence, given that the other relevant data coincides since the thread is the same, these two binary sequences have opposite sign and their contributions cancel.

For Move II, the tangency between a tine and an edge $e$, the left hand side has $|\SB_{1,1}|=|\SB_{0,0}|=1$ given by the constant sequences and else $\SB_{1,0}=\SB_{0,1}=\emptyset$. On the right hand side, there is a unique sequence in each $\SB_{1,1}$ and $\SB_{0,0}$, which switches at both of the intersection points $q_1,q_2\in e$: these sequences are counted with the factor $H(B, q_1)\cdot H(B, q_2)=(\pm e^{\pm 1})(\pm e^{\mp1})=1$ and thus their contributions coincides with that of the constant functions. Certainly, $\SB_{1,0}=\SB_{0,1}=\emptyset$ also holds.

For Move III, where a tine crosses along a vertex, the counts of unsigned contributions read as follows:
\begin{center}
  \begin{tabular}{| l | l | l |}
    \hline
    Fig.~\ref{fig:VertexCrossing} & Left & Right \\ \hline
    $1\lr1$ & $e_1^{-1}e_3$ & $e_2^{-1}\cdot (e_1^{-1}e_2e_3)$ \\ \hline
    $1\lr0$ & $e_1^{-1}\cdot(e_2^{-1}e_1e_3)\cdot e^{-1}_3$ & $e_2^{-1}$ \\ \hline
    $0\lr1$ & $0$ & $e_2+(e_3^{-1}e_1e_2)\cdot e_2^{-1}\cdot(e_1^{-1}e_2e_3)$ \\ \hline
    $0\lr0$ & $e_1e_3^{-1}$ & $(e_3^{-1}e_1e_2)\cdot e_2^{-1}$\\
    \hline
  \end{tabular}
\end{center}
The counts of signed contributions are verified in the tables of Appendix \ref{app:signs}.
\end{proof}

\begin{proof}[Proof of Proposition \ref{prop:rake indep}:]
By Lemma \ref{lem:easy rake moves} it suffices to define a dg--algebra isomorphism

$$\p: (\wt\A_G, \wt\dd_{\Gamma_0})\lr(\wt\A_G, \wt\dd_{\Gamma_1}),$$

where the two rakes $R_0$ and $R_1$ differ by a single tine switch as depicted in Figure \ref{fig:TineSwitch}: the tines $\gamma^0_{f_m}$ and $\gamma^0_{f_n}$ switch resulting in two new parallel tines $\gamma^1_{f_n},\gamma^1_{f_m}$. Because $\gamma_{f_m}$ and $\gamma_{f_n}$ are parallel, we can consider a curve $\gamma_{m,n}$ which lies between the tines $\gamma_{f_m}$ and $\gamma_{f_n}$ and passes through $c_{f_m}$ and $c_{f_n}$. The curve $\gamma_{m,n}$ (depicted in Figure \ref{fig:ProofTineSwitch}) should be considered as a generalized tine which governs the transition of binary sequences when a tine switch occur. That is, the count of binary sequences along $\gamma_{m,n}$ will define $\p$. Notice that we have $E(\gamma_{m,n}) = E(\gamma_{f_m}) = E(\gamma_{f_n})$, and $\alpha(\gamma_{m,n}) = \alpha(\gamma_{f_m}) \cap \alpha(\gamma_{f_n})$.

\begin{figure}[h!]
\centering
  \includegraphics[scale=0.65]{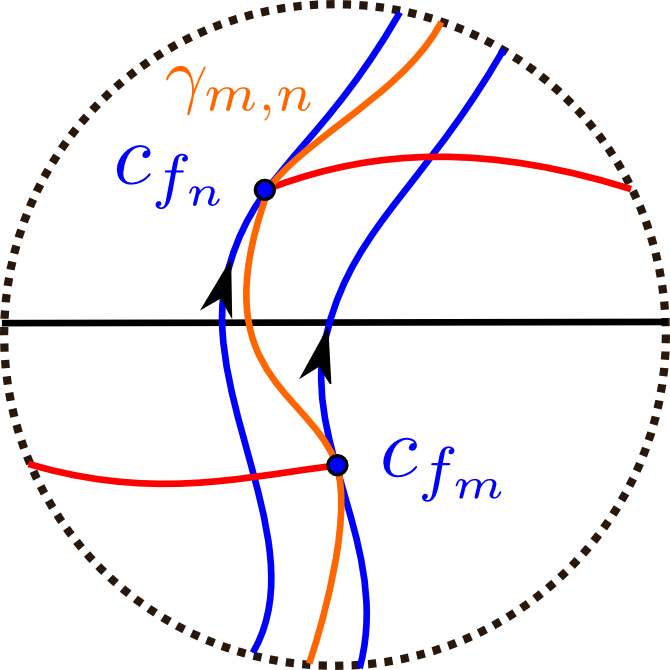}
  \caption{The curve $\gamma_{m,n}$, which accounts for the transition in a tine switch.}\label{fig:ProofTineSwitch}  
\end{figure}

Following Definition \ref{def:binary} a binary sequence $B$ along $\gamma_{m,n}$ is a binary sequence which changes value at all points of $E(\gamma_{m,n})$, changes its value from $0$ to $1$ at $c_{f_m}$ and $c_{f_n}$, and either remains constant or change its value from $0$ to $1$ at points of $\alpha(\gamma_{m,n})$. By considering the set of points $\alpha(B) \sse \alpha(\gamma_{m,n})$ where $B$ switches its values at a thread, we can define

$$H(B) := \prod_{q \in E(\gamma_{f_{m,n}})} H(B,q)\cdot\prod_{p \in \alpha(B)} \sigma(p) H(\tau_p).$$

Let $\mathcal{B}_{a,b}(\gamma_{m,n})$ be the set of all binary sequences along $\gamma_{m,n}$ with initial value $a$ and final value $b$. We define a map $\p: \wt\A_G\lr\wt\A_G$ which is the identity on all edges $e_i \in E$ and faces $f_j \in F$, and acts as follows on the degree-2 generators:
\begin{equation}\label{eqn:tine iso}
\p(x) = x + \left(\sum_{B \in \mathcal{B}_{1,1}(\gamma_{m,n})}H(B)\right)f_m f_n,\end{equation}
$$\p(y) = y + \left(\sum_{B \in \mathcal{B}_{1,0}(\gamma_{m,n})}H(B)\right)f_m f_n,$$
$$\p(z) = z + \left(\sum_{B \in \mathcal{B}_{0,0}(\gamma_{m,n})}H(B)\right)f_m f_n.$$
We claim that the map $\p$ commutes with the differentials: $\p \circ \wt\dd_{\Gamma_0} = \wt\dd_{\Gamma_1} \circ \p$, which we prove now.

First, the action of the two differentials $\wt\dd_{\Gamma_0}$, $\wt \dd_{\Gamma_1}$ coincides on all the faces $f_j \in F$ and $e_i \in E$ since their definitions are independent of the choice of rakes. On the faces denote $\wt\dd=\wt\dd_{\Gamma_0}=\wt\dd_{\Gamma_1}$. For the degree-2 generators $x,y,z\in(\wt \A_G)_2$, the condition $\p \circ \wt\dd_{\Gamma_0} = \wt\dd_{\Gamma_1} \circ \p$ is equivalent to the following conditions:
$$\wt\dd_{\Gamma_0} x - \wt \dd_{\Gamma_1} x = \left(\sum_{B \in \mathcal{B}_{1,1}(\gamma_{m,n})}H(B)\right)\left(\left(\wt\dd f_m\right) f_n - \left(\wt\dd f_n\right)f_m\right)$$
$$\wt\dd_{\Gamma_0} y - \wt \dd_{\Gamma_1} y = \left(\sum_{B \in \mathcal{B}_{1,0}(\gamma_{m,n})}H(B)\right)\left(\left(\wt\dd f_m\right) f_n - \left(\wt\dd f_n\right)f_m\right)$$
$$\wt\dd_{\Gamma_0} z - \wt \dd_{\Gamma_1} z = \left(\sum_{B \in \mathcal{B}_{0,0}(\gamma_{m,n})}H(B)\right)\left(\left(\wt\dd f_m\right) f_n - \left(\wt\dd f_n\right)f_m\right).$$
Let $w\in\{x,y,z\}$, the only terms which contribute to $\wt\dd_{\Gamma_0} w - \wt \dd_{\Gamma_1} w$ come from binary sequences along $\gamma^0_{f_m}$, $\gamma^1_{f_m}$, $\gamma^0_{f_n}$, and $\gamma^1_{f_n}$, since all other terms will appear equally in $\wt \dd_{\Gamma_0} w$ and $\wt \dd_{\Gamma_1} w$. Therefore, we can write $\wt\dd_{\Gamma_0} w - \wt \dd_{\Gamma_1} w = P^w_mf_m + P^w_nf_n$, and it remains to show
$$P^w_m = -\left(\sum_{B \in \mathcal{B}_{a,b}(\gamma_{m,n})}H(B)\right)\left(\wt\dd f_n\right),\quad P^w_n = \left(\sum_{B \in \mathcal{B}_{a,b}(\gamma_{m,n})}H(B)\right)\left(\wt\dd f_m\right)$$
where $(a,b) = (1,1)$, $(1,0)$, or $(0,0)$ according to $w = x, y$, or $z$, respectively.

Consider a binary sequence $B^0$ along $\gamma^0_{f_m}$. If it does not switch its value at a thread inside the face $f_n$, then it has a counterpart binary sequence $B^1$ along $\gamma^1_{f_m}$ obtained by switching in all the corresponding locations as $B^0$; hence these two contributions $H(B^0) = H(B^1)$ cancel in $\wt\dd_{\Gamma_0} w - \wt \dd_{\Gamma_1} w$. In consequence, the only binary sequences $B$ which contribute to $P_m^w$ are those sequences along $\gamma^0_{f_m}$ or $\gamma^1_{f_m}$ which necessarily switch their value at a thread inside $f_n$. Thus $B$ defines a binary sequence along $\gamma_{m,n}$: if $B$ is a binary sequence along $\gamma^0_{f_m}$ which switches value at a thread near to $f_n$, we can use the same values to define a sequence along $\gamma_{m,n}$ which switches value at $c_{f_n}$.

The data lost in this association accounts for \emph{which} thread inside $f_n$ is the one in which the binary sequence $B$ switches, which is counted by the second factor $\wt\dd f_n$ in the expression for $P^w_m$. It is readily seen that each choice of thread appears once in either $\wt\dd_{\Gamma_0} w$ or $\wt\dd_{\Gamma_1} w$, but not both, and in either case the signs give the correct contributions. This proves the expansion for $P^w_m$, and the proof for $P^w_n$ is identical.
\end{proof}

Proposition \ref{prop:rake indep} shows the independence of the dg--algebra isomorphism type from the choice of rake. Another element of a garden is the choice of a web. Clearly, in each face the choice of threads is essentially unique, since each face is contractible. The exception is the choice of threads at infinity, which we address here.

\begin{prop}\label{prop:anchor indep}
Let $G$ be a graph and $\Gamma_0=\{\{c_f\},\{\gamma_f\},\widehat W^0\}$, $\Gamma_1=\{\{c_f\},\{\gamma_f\},\widehat W^1\}$ two gardens such that $W^0=W^1$. Then there exists a dg--algebra isomorphism
$$\p:(\wt\A_G, \wt\dd_{\Gamma_0})\lr(\wt\A_G, \wt\dd_{\Gamma_1}),$$
which restricts to the identity on $\Lambda_G \sse \wt\A_G$.
\end{prop}

\begin{proof}
Apply Proposition \ref{prop:rake indep} to ensure the rake $\{\gamma_f\}$ is such that all the exterior vertices of $G$ lie to the right of every tine in the rake. Thus we may assume that $\widehat W^0\setminus W^0$ and $\widehat W^1\setminus W^1$ differ in a bottom--top move as in Figure \ref{fig:ProofAnchor}, since $\widehat W^0$ and $\widehat W^1$ must differ by a finite sequence of such moves.

\begin{figure}[h!]
\centering
  \includegraphics[scale=0.5]{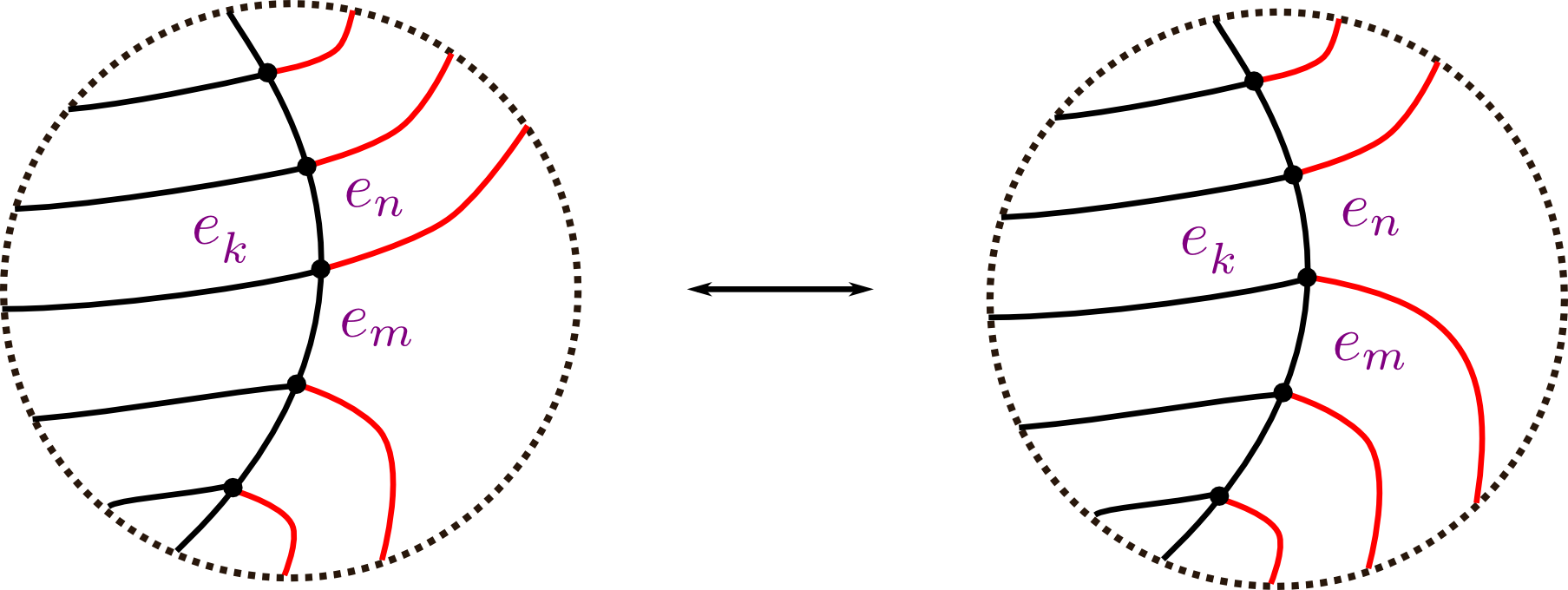}
  \caption{Bottom--top move for the set of threads at infinity.}\label{fig:ProofAnchor}  
\end{figure}

Consider the algebra isomorphism $\p: \wt \A_G \to \wt \A_G$ defined as the identity on all faces and edges and acting in the degree-2 generators as
\begin{equation}\label{eqn:anchor iso}
\p(x) = x+ e_me_ne^{-1}_k y,\quad\p(z) = z + e_me_ne^{-1}_k y,\quad \p(y) = y.\end{equation}

It is readily verified that $\p\circ\wt\dd_{\Gamma_0} = \wt\dd_{\Gamma_1} \circ \p$.
\end{proof}

Let us now use Propositions \ref{prop:rake indep} and \ref{prop:anchor indep} to prove the following result.

\begin{thm}\label{thm:garden indep}
Let $G$ be a planar cubic graph equipped with two gardens $\Gamma_0,\Gamma_1$. Then there exists a graded algebra isomorphism
$\p:(\wt\A_G, \wt\dd_{\Gamma_0})\lr(\wt\A_G, \wt\dd_{\Gamma_1})$
such that $\p\circ\wt\dd_{\Gamma_0} = \wt\dd_{\Gamma_1} \circ \p$.
\end{thm}

\begin{proof}
Choosing a different centering and interior threads yields to combinatorially equivalent configurations, and Propositions \ref{prop:rake indep} and \ref{prop:anchor indep} show that the choices of different rakes and threads at infinity yield isomorphic algebras with an isomorphism commuting with the $\wt \dd$--operators.

It remains to show that choosing a different orientation induces an isomorphism. It suffices to show this for two orientations that differ in exactly one edge $e_i \in E$, and let $\wt\dd_0$ and $\wt\dd_1$ be the chain differentials corresponding to the two choices of orientation. The ring isomorphism $\phi: \Lambda_G \lr \Lambda_G$ which takes $e_i$ to $-e_i$ and is the identity on all other generators induces an algebra isomorphism $\p : \wt\A_G \lr \wt\A_G$, and it is easily checked that it satisfies $\p\circ\wt\dd_{\Gamma_0} = \wt\dd_{\Gamma_1} \circ \p$.
\end{proof}

\section{Proof of $\wt\dd^2 = 0$} \label{sec:dsquared}

Let $G=(V,E,F)$ be a graph equipped with a garden $\Gamma$, this section is devoted to showing that $(\wt\A_G,\wt\dd_\Gamma)$ is a dg--algebra, that is the identity $\wt\dd_\Gamma^2=0$ is satisfied.

\begin{thm}\label{thm:dsquared}
Let $(G,\Gamma)$ be a decorated graph equipped with a garden $\Gamma$. Then $\wt\dd_\Gamma^2=0$.
\end{thm}

\begin{proof}
It suffices to prove $\wt\dd^2_\Gamma w=0$ for $w\in \{x,y,z\}$, since $\wt\dd_\Gamma = 0$ on the coefficient ring $\Lambda_G = (\wt\A_G)_0$. Each element $\wt\dd_\Gamma w\in(\wt\A_G)_1$ is a finite sum of terms of the form $H(B) f$, where $B$ is a binary sequence along $\gamma_f$. Thus the terms of $\wt\dd^2_\Gamma w$ are of the form $H(B)\cdot H(\tau_f(v))$, where $B$ is a binary sequence along $\gamma_f$ and $\tau_f(v)$ is the thread inside $f$ for some $v\in f$. This is depicted in Figure \ref{fig:TermDel}.

\begin{figure}[h!]
\centering
  \includegraphics[scale=0.65]{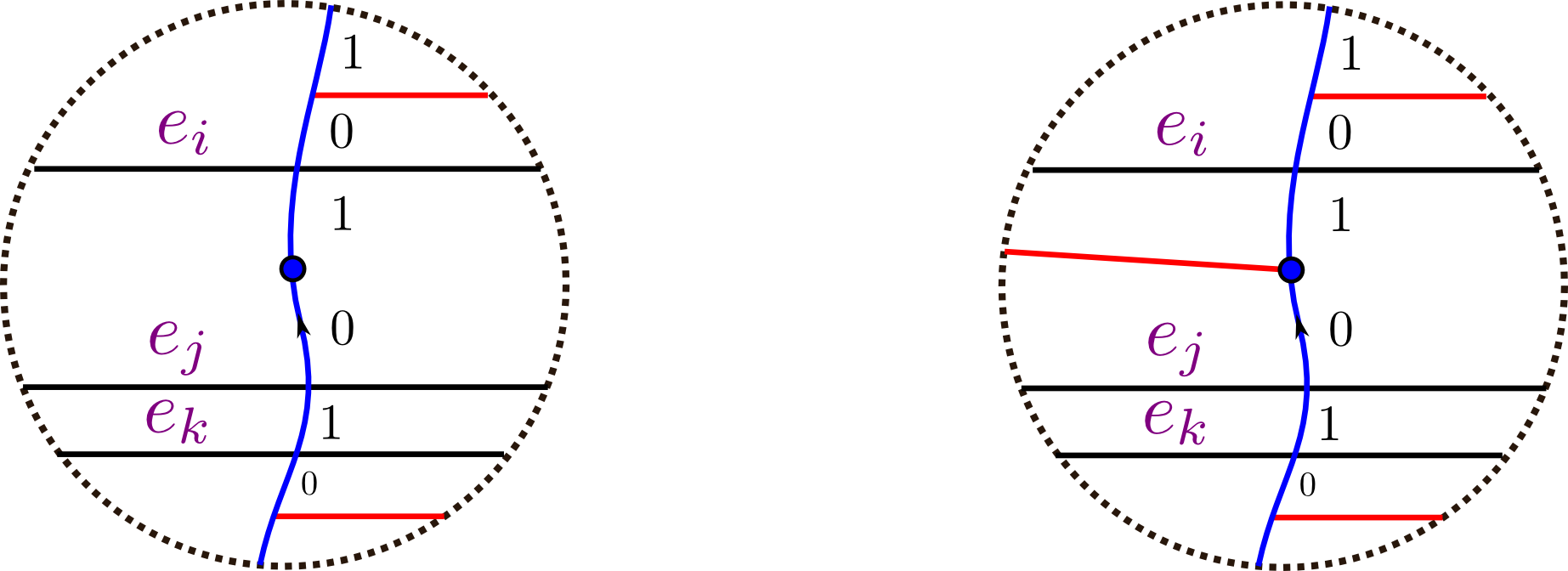}
  \caption{Part of a binary sequence defining a term of $\wt\dd_\Gamma w$, on the left, and the data defining a term of $\wt\dd^2_\Gamma w$, on the right. Both depicted near the center of the tine.}\label{fig:TermDel}
\end{figure}

First, consider a foliation $\SL=\{\SL_t\}$ of the square $[-1,1]\times[-1,1]$ such that:

\begin{itemize}
 \item[1.] The space of leaves of $\SL$ is smoothly parametrized by $t\in[-1,1]$, and each leaf $\SL_t$ is an embedded path with respective endpoints on $[-1, 1] \x \{-1\}$ and $[-1,1] \x \{1\}$. In addition, $\SL_{-1}=\{-1\}\times[-1,1]$ and $\SL_{1}=\{1\}\times[-1,1]$, and for all $f\in F$, the tines $\gamma_f$ are leaves $\SL_{t_f}$ of $\SL$.\\
 \item[2.] For all but finitely many $t \in [-1, 1]$, $\SL_t\cap \Int(\widehat W)$ and $\SL_t\cap E$ consist of finite sets of transverse intersection points. At these finitely many exceptional $t$, the finitely many intersection points in $\SL_t\cap \Int(\widehat W)$ and $\SL_t\cap E$ are also allowed to be non-oscillatory tangencies.
\end{itemize}

There exist foliations satisfying the first condition and we can construct a foliation satisfying the second condition by a $C^\infty$-small perturbation.

Now, a leaf $\SL_t$ of the foliation $\SL$ is said to be regular if $t \in (-1, 1)$, $\SL_t\neq\gamma_f$ for any $f\in F$, $\SL_t\cap V=\emptyset$, and the intersection points in the sets $\a(\SL_t) := \SL_t\cap \Int(\widehat W)$ and $E(\SL_t) := \SL_t\cap E$ are transverse. A critical leaf $\SL_t$ is by definition a leaf which is not regular, and the set of critical leaves is finite.

\begin{remark}
By a $C^\infty$-small perturbation, we assume that each critical leaf is not regular for a unique reason: it either contains a single vertex, has a unique tangency with an edge, has a unique tangency with a thread, or is equal to a tine.\hfill$\Box$
\end{remark}

Let $n(\SL)\in\N$ be the number critical leaves, and choose a set $\{t_1, \ldots, t_{n(\SL)+1}\} \sse (-1, 1)$ such that each $\SL_{t_l}$ is regular and lies between the $l$th and $(l+1)$st critical leaves; these leaves $\SL_{t_l}$ are referred to as the standard leaves and denoted by $\SL_l$ for $1\leq l\leq n(\SL)+1$.

Consider the set $\mathcal{B}^{\op{reg}}_{a,b}(\SL)$ of binary sequences along all standard leaves whose initial and final values are $a$ and $b$ respectively, and the set $\mathcal{B}_{a,b}^{\op{*}}$ of pairs $(B, \tau_f(v))$ where $B \in \mathcal{B}_{a,b}(f)$ and $\tau_f(v)$ is a thread from $c_f$; let us denote $\mathcal{B}_{a,b}(\SL) = \mathcal{B}^{\op{reg}}_{a,b}(\SL) \cup \mathcal{B}_{a,b}^{\op{*}}$.

In order to prove $\wt\dd_\Gamma^2 = 0$, we construct an involution

$$\Psi:\mathcal{B}^{\op{*}}_{a,b} \lr \mathcal{B}^{\op{*}}_{a,b},\qquad (a,b)\in\{(1,1),(1,0),(0,0)\},$$

such that $H(\Psi(B, \tau_f(v))) = - H(B)\cdot H(\tau_f(v))$; these involutions are defined as follows.

For each $B_0 = (B, \tau_f(v)) \in \mathcal{B}^{\op{*}}_{a,b}$ we construct a finite sequence $B_k \in \mathcal{B}_{a,b}(\SL)$, $0\leq k\leq N=N(B_0)$ with the following properties:

\begin{itemize}
\item[-] $B_N \in \mathcal{B}^{\op{*}}_{a,b}$, and $B_k \in \mathcal{B}^{\op{reg}}_{a,b}(\SL)$ for $1\leq k\leq N-1$,
\vspace{0.1cm}
\item[-] $H(B_0) = -H(B_N)$ and $H(B_k) = \pm H(B_{k'})$ for all $k, k'$,
\vspace{0.1cm}
\item[-] The binary sequence $B_k$ only depends on $B_{k-1}$ and the sign of $H(B_0)$,
\vspace{0.1cm}
\item[-] The sequence $\{B'_k\}$ defined by the initial condition $B'_0 = B_N$ is given by $B'_k = B_{N-k}$.
\end{itemize}

The involution $\Psi:\mathcal{B}^{\op{*}}_{a,b} \lr \mathcal{B}^{\op{*}}_{a,b}$ is defined as $\Psi(B_0) = B_N$ which proves $\wt\dd^2_\Gamma=0$. Let us construct the sequence $\{B_k\}$ from any initial condition $(B, \tau_f(v))\in\mathcal{B}^{\op{*}}_{a,b}$.

Consider the element $B_0 = (B, \tau_f(v))$ and let $\SL_l$ be the standard leaf which is adjacent to the tine $\gamma_f$ and lies to its right or its left according to whether the thread $\tau_f(v)$ points to the right or to the left of the tine $\gamma_f$. The first element $B_1$ in the sequence is defined to be the binary sequence along $\SL_l$ which switches its value at the same edges and threads as $B$, except that it additionally switches from $0$ to $1$ at $p \in \tau_f(v)$ instead of at the center $c_f$. Notice that the contributions $H(B_1) = \pm H(B_0)\cdot H(\tau_f(v))$ coincide up to sign since the same factors appear in both expressions with the exception of the contribution of $\sigma(p)$ combing from the value switch at $q$; and thus, $H(B_1) = H(B_0)$ if $\tau_f(v)$ lies to the right of the tine $\gamma_f$, and $H(B_1) = - H(B_0)$ if it lies to its left. This defines the binary sequence $B_1\in\mathcal{B}^{\op{reg}}_{a,b}$.

Suppose that the binary sequence $B_n\in\mathcal{B}^{\op{reg}}_{a,b}$ is defined, $B_n$ is supported along the standard leaf $\SL_j$ and $H(B_n) = \sigma_n H(B_0)$ for the correct sign $\sigma_n \in \{-1, 1\}$. Let $\SL^{\op{crit}}$ be the critical leaf lying to the right or to the left of the regular leaf $\SL_j$ according to whether $\sigma_n=1$ or $-1$ respectively. The binary sequence $B_{n+1}$ is now defined depending on the type of singularity presented by the leaf $\SL^{\op{crit}}$.

{\bf Case A:} The singular leaf $\SL^{\op{crit}}$ is tangent to an edge $e \in E$.

In this case, the regular leaf $\SL_{j+\sigma_n}$ differs from $\SL_j$ only in two additional intersections with the edge $e$, as depicted in Figure \ref{fig:d2EdgeCrossing}. Then $B_{n+1}$ is the unique binary sequence along $\SL_{j+\sigma_n}$ such that outside of the neighborhood where the tangency occurs it satisfies $B_{n+1} = B_n$. The contributions might only different in the two additional intersections, but since these are consecutive they contribute the factor of $(\pm e^{\pm 1})(\pm e^{\mp 1})=1$ to either $H(B_n)$ or $H(B_{n+1})$, and consequently $H(B_{n+1}) = H(B_n)$.

\begin{figure}[h!]
\centering
  \includegraphics[scale=0.75]{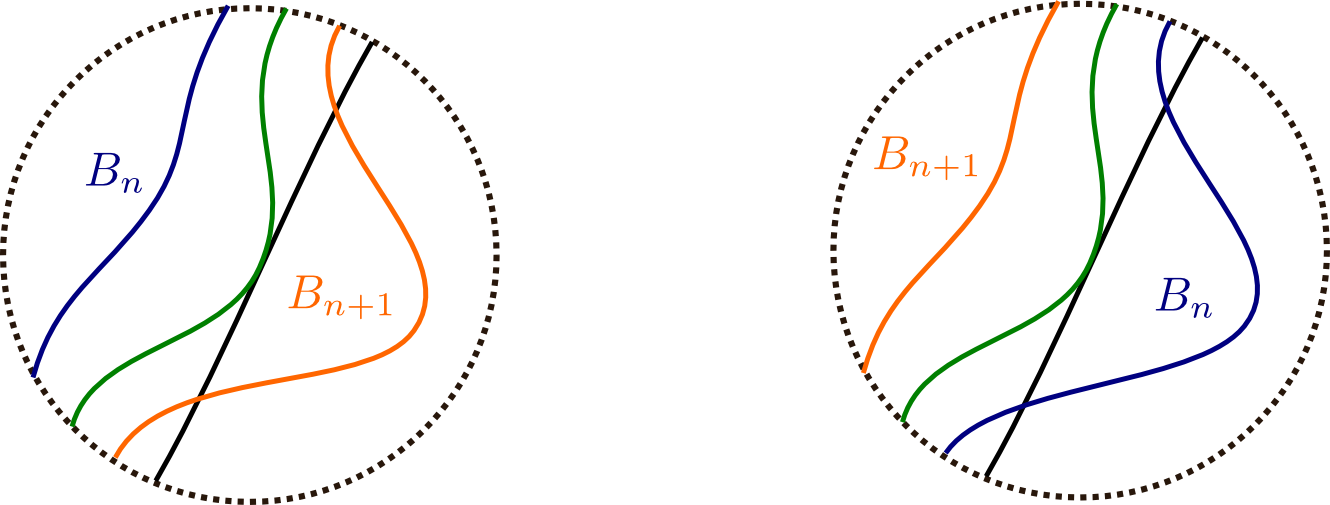}
  \caption{The sequence $\{B_n\}$ across an edge tangency.}\label{fig:d2EdgeCrossing}
\end{figure}

{\bf Case B:} The singular leaf $\SL^{\op{crit}}$ is tangent to a thread $\tau$.

The regular leaf $\SL_{j+\sigma_n}$ differs from $\SL_j$ in two additional intersections $q_1$ and $q_2$ with the thread $\tau$, which is depicted in Figure \ref{fig:d2ThreadCrossing}. First, in case $q_1,q_2$ lie on $\SL_{j+\sigma_n}$, we define $B_{n+1}$ to be the binary sequence along $\SL_{j+\sigma_n}$ which does not change value at $q_1$ and $q_2$, and is equal to $B_n$ outside of a small neighborhood of the tangency. Second, in case $q_1$ and $q_2$ lie on $\SL_j$ but the binary sequence $B_n$ happens to be constant at those points, we define $B_{n+1}$ to be the binary sequence along $\SL_{j+\sigma_n}$ equal to $B_n$; note that in these two cases $H(B_{n+1}) = H(B_n)$.

Third, in case the binary sequence $B_n$ switches its value from $0$ to $1$ at either $q_1$ or $q_2$, for definiteness let us assume $q_1$, we define $B_{n+1}$ to be the binary sequence along $\SL_j$ which is equal to $B_n$ except that it is constant at the point $q_1$ and switches value at $q_2$. Note that in this case, the equality $\sigma(q_1) = -\sigma(q_2)$ implies $H(B_{n+1})=-H(B_n)$.

\begin{figure}[h!]
\centering
  \includegraphics[scale=0.75]{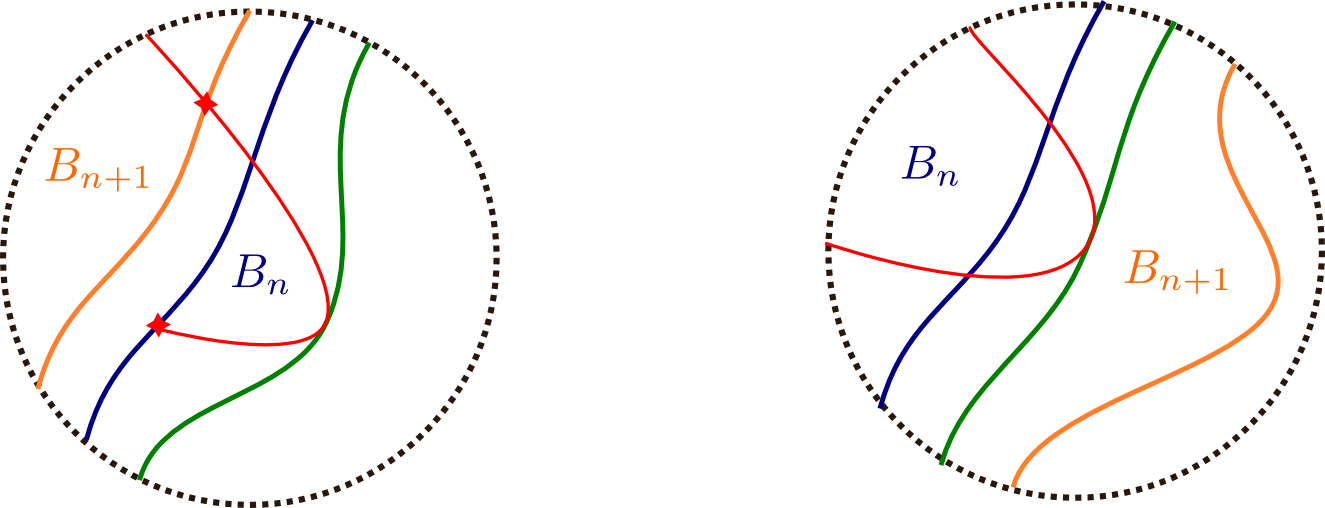}
  \caption{The sequence $\{B_n\}$ at a tine-thread singularity.}\label{fig:d2ThreadCrossing}
\end{figure}

{\bf Case C:} The singular leaf $\SL^{\op{crit}}$ contains a vertex.

In the three cases where near the vertex the initial and final values of the binary sequence $B_n$ are $1 \lr 1$, $1 \lr 0$, or $0 \lr 0$, there is a unique binary sequence $B_{n+1}$ defined along $\SL_{j+\sigma_n}$ which equals $B_n$ outside of this neighborhood: in these cases $H(B_{n+1}) = H(B_n)$, as depicted in Figure \ref{fig:d2VertexCrossing}. In the fourth case, where the boundary conditions are $0 \lr 1$, there are two binary sequences along $\SL_j$ which are equal outside of the local neighborhood, and these have opposite signs. In that scenario we define the binary sequence $B_{n+1}$ to be the unique sequence along $\SL_j$ distinct from $B_n$ with contribution $H(B_{n+1}) = - H(B_n)$. This is depicted in Figure \ref{fig:d2VertexRetro}.

\begin{figure}[h!]
\centering
  \includegraphics[scale=0.7]{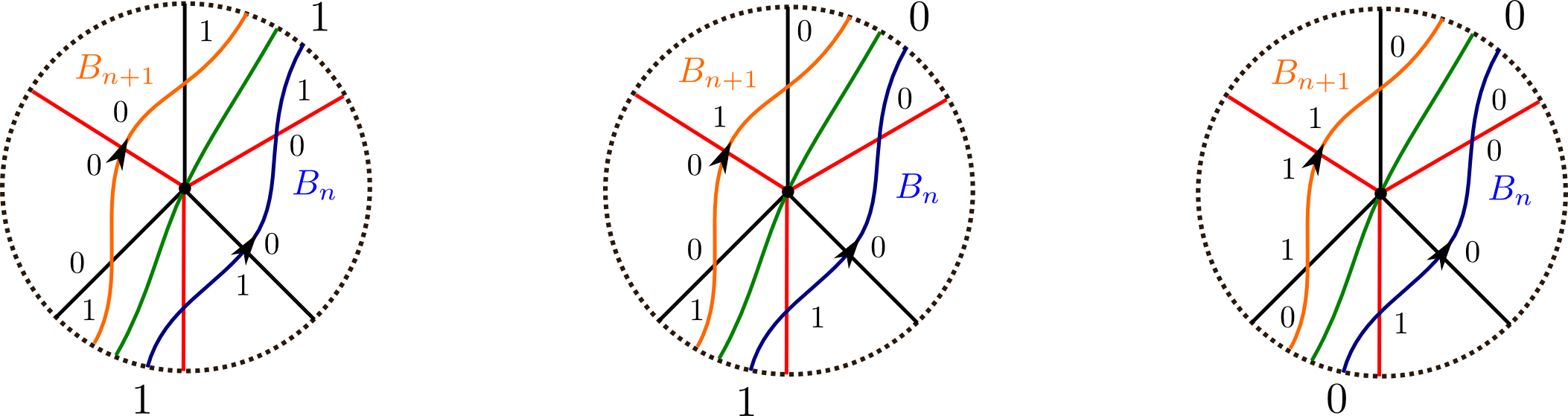}
  \caption{Transition to $B_{n+1}$ near a vertex crossing in three cases. The singular leaf (green) is resolved according to the invariance principle.}\label{fig:d2VertexCrossing}
\end{figure}

\begin{figure}[h!]
\centering
  \includegraphics[scale=0.75]{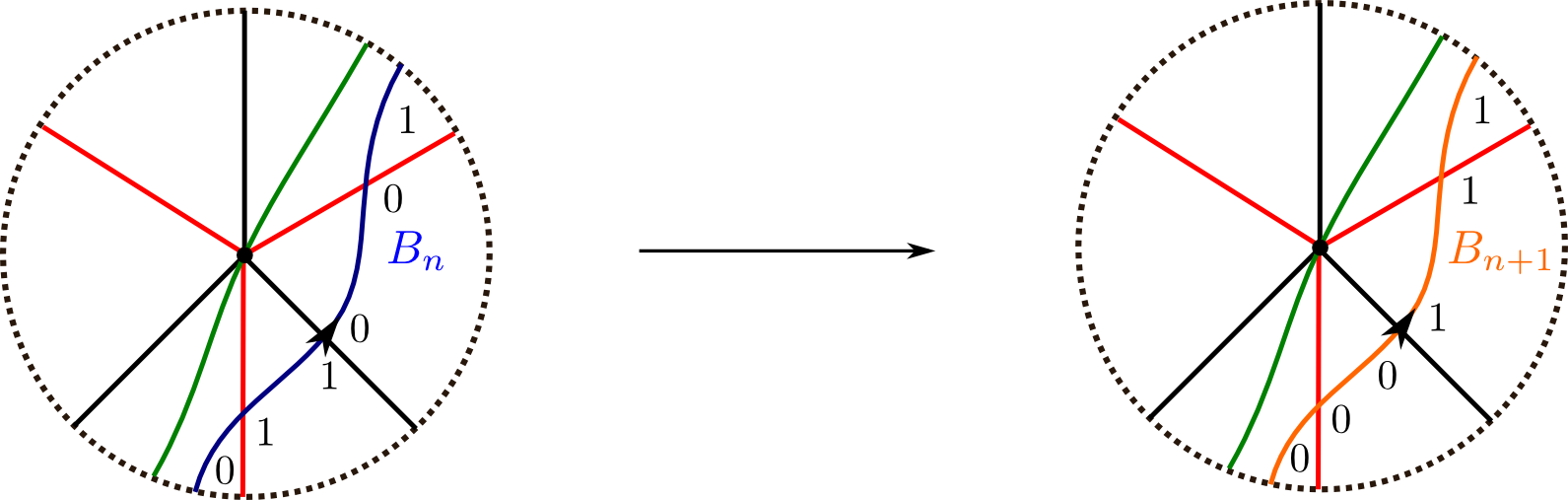}
  \caption{The retrograde transition $B_{n+1}$ near a vertex crossing.}\label{fig:d2VertexRetro}
\end{figure}

{\bf Case D:} The singular leaf is $\SL^{\op{crit}} = \{\pm 1\} \x [-1, 1]$.

In this case either $j=1$ and $\sigma_n = -1$ or $j = n(\SL)+1$ and $\sigma_n = 1$. In the latter case, the regular leaf $\SL_j$ does not intersect any edges or threads and therefore $B_n$ is constant; then let $B_{n+1}$ be the binary sequence along $\SL_1$ which is equal to the same constant. In the former case, the regular leaf $\SL_j$ does not intersect any edges but it does intersect all exterior threads. Since we only consider the cases $(a,b) = (1,1)$, $(1,0)$, or $(0,0)$, it follows that $B_n$ must also be constant: the sequence $B_{n+1}$ is the binary sequence along $\SL_{n(\SL)+1}$ with that same constant. In both of these cases, we have $H(B_{n+1}) = H(B_n) = 1$.

{\bf Case E:} The singular leaf $\SL^{\op{crit}} = \gamma_f$ is the tine for a face $f\in F$.

There are two possible depending on whether $B_n$ switches value on a thread at a point near $c_f$. In case the binary sequence $B_n$ is constant near the center $c_f$, we define $B_{n+1}$ to be the binary sequence along $\SL_{j+\sigma_n}$ which is also constant near $c_f$ and equal to $B_n$; it readily follows that $H(B_{n+1}) = H(B_n)$.

In the case that $B_n$ switches its value from $0$ to $1$ at a point $q$ near a center $c_f$ on the thread $\tau_f(v)$, we define $B_{n+1}=(B^*,\tau_f(v))\in \mathcal{B}^{\op{*}}_{a,b}$ to be the binary sequence $B^*$ along the tine $\gamma_f$ which switches its value at $c_f$ and is otherwise equals $B_n$. Notice that $\sigma(q) = -\sigma_n$ since $\gamma_f$ lying to the right of the regular leaf $\SL_j$ implies that the thread $\tau_f(v)$ lies to the left of the tine $\gamma_f$, and vice versa. In consequence the contribution satisfies $H(B_{n+1}) = \sigma(q) H(B_n) = -\sigma_n H(B_n) = - H(B_0)$.

This completes the definition of the involution $\Psi:\mathcal{B}^{\op{*}}_{a,b} \lr \mathcal{B}^{\op{*}}_{a,b}$. It is readily seen that switching the sign of $\sigma_n$ at a step reverses the sequential process. In particular, this implies that the sequence $\{B_n\}$ must eventually lie in the critical set $\mathcal{B}^{\op{*}}_{a,b}$, since otherwise we would have an infinite sequence of distinct elements in a finite set; this also shows that whenever $B'_0 = B_N$ we must also have $B'_n = B_{N-n}$, which shows that $\Psi:\mathcal{B}^{\op{*}}_{a,b} \lr \mathcal{B}^{\op{*}}_{a,b}$ is an involution.\end{proof}

\section{Computations}\label{sec:app}

In this section we shall be using the notation $\e_{a_1\cdots a_k,b_1\cdots b_l}=e_{a_1}\cdots e_{a_k}e^{-1}_{b_1}\cdots e_{b_l}^{-1}$.

\subsection{4-vertex graphs}\label{ssec:chek} Let us start with the two planar cubic graphs with genus $g=1$. First, consider the graph $G_1$ at the left of Figure \ref{fig:g1graphs} with the depicted choice of garden, which corresponds to the Legendrian front of the Chekanov torus. Note that we only depict the threads which intersect a tine on their interior. In this case $g=1$ and thus the dg-algebra is to be generated as an $\F[e_1^{\pm1},\ldots,e_6^{\pm1}]$-algebra by the degree-2 elements $x,y,z\in(\wt\A_{G_1})_2$ and $f_1,f_2,f_3\in(\wt\A_{G_1})_1$.

\begin{figure}[h!]
\centering
  \includegraphics[scale=0.65]{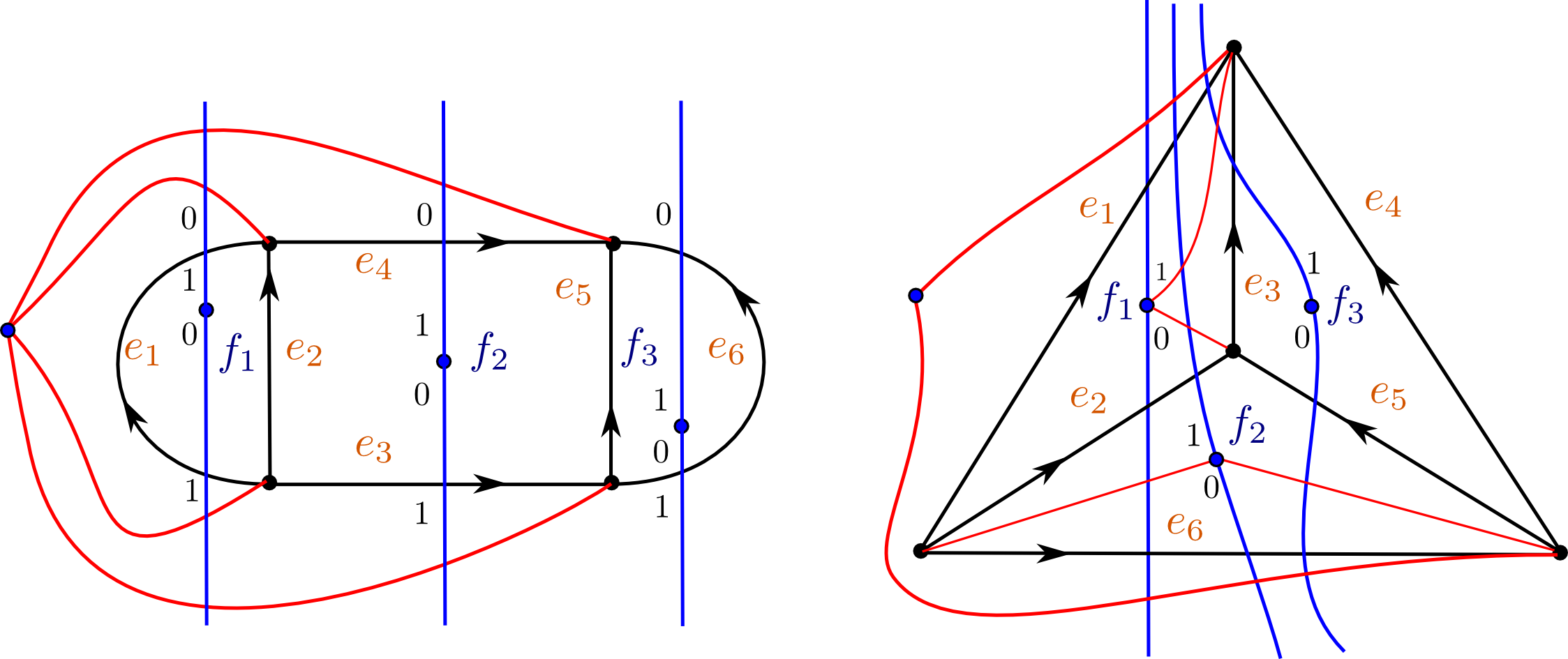}
  \caption{Planar cubic graphs with $g=1$ and gardens.}\label{fig:g1graphs}
\end{figure}

The differential is given by the count of binary sequences, which by this specific choice are restricted to the values depicted in Figure \ref{fig:g1graphs}. The count yields the following result:

\begin{align*}
\wt\dd_\Gamma x &= -e^{-2}_1(-\e_{14,2}+\e_{46,5})\cdot f_1+e_3^{-1}e_4^{-1}\e_{46,5}\cdot f_2,\\
\wt\dd_\Gamma y &= -e^{-2}_1\cdot f_1+e_3^{-1}e_4^{-1}\cdot f_2-e_6^{-2}\cdot f_3,\\
\wt\dd_\Gamma z &= (-\e_{13,2}+\e_{36,5})(-e^{-2}_1)\cdot f_1+\e_{36,5}e_3^{-1}e_4^{-1}\cdot f_2,
\end{align*}

and it is readily computed that $\wt\dd_\Gamma$ is indeed a differential:

\begin{align*}
\wt\dd^2_\Gamma x &= -e^{-2}_1(-\e_{14,2}+\e_{46,5})(-\e_{12,4}-\e_{12,3})+e_3^{-1}e_4^{-1}\e_{46,5}(\e_{45,6}+\e_{35,6}-\e_{24,1}-\e_{23,1})=0,\\
\wt\dd^2_\Gamma y &= -e^{-2}_1(-\e_{12,4}-\e_{12,3})+e_3^{-1}e_4^{-1}(\e_{45,6}+\e_{35,6}-\e_{24,1}-\e_{23,1})-e_6^{-2}(\e_{56,4}+\e_{56,3})=0,\\
\wt\dd^2_\Gamma z &= (-\e_{13,2}+\e_{36,5})(-e^{-2}_1)(-\e_{12,4}-\e_{12,3})+\e_{36,5}e_3^{-1}e_4^{-1}(\e_{45,6}+\e_{35,6}-\e_{24,1}-\e_{23,1})=0.
\end{align*}

This computation is particularly simple due to the choice of garden, and the reader is invited to study the sequences $\{B_n\}$ appearing in the proof of Theorem \ref{thm:dsquared}, which in this case can be quickly done by hand.

Second, consider the graph $G_2$ shown in the right of Figure \ref{fig:g1graphs}. From the Legendrian viewpoint it corresponds to the Clifford torus, which is part of the mirror of the pair-of-pants \cite{Na,TZ}, and it arises as the critical graph of $SU(2)$ supersymmetric QCD with $N_f=4$ quarks \cite{EHIY,Lo}. Let us compute the differential in the dg--algebra structure:

\begin{align*}
\wt\dd_\Gamma x &= (-e_6^{-1})\e_{26,1}(-e_2^{-1})(-e_1^{-1})\e_{14,3}\cdot f_1+(-e_6^{-1})(-e_2^{-1})(e_1+(-\e_{23,5}+\e_{13,4})(-e_1^{-1})\e_{14,3})\cdot f_2\\
&\qquad +(-e_6^{-1})\e_{56,4}e_5^{-1}(-e_3^{-1})(e_1+\e_{13,4}(-e_1^{-1})\e_{14,3})\cdot f_3\\
& = -\e_{4,13}f_1 + \e_{4,56}f_2,\\
\wt\dd_\Gamma y &= (-e_6^{-1})\e_{26,1}(-e_2^{-1})(-e_1^{-1})\cdot f_1 + (-e_6^{-1})(-e_2^{-1})(-\e_{23,5}+\e_{13,4})(-e_1^{-1})\cdot f_2\\
&\qquad +(-e_6^{-1})\e_{56,4}e_5^{-1}(-e_3^{-1})\e_{13,4}(-e_1^{-1})\cdot f_3 \\
&=-e_1^{-2} f_1 + (\e_{3,156} - \e_{3,246})f_2 - e_4^{-2}f_3,\\
\wt\dd_\Gamma z &= (\e_{46,5}(-e_6^{-1})\e_{26,1} + e_6)(-e_2^{-1})(-e_1^{-1})\cdot f_1 + \e_{46,5}(-e_6^{-1})(-e_2^{-1})(-\e_{23,5}+\e_{13,4})(-e_1^{-1})\cdot f_2\\
&\qquad + (\e_{46,5}(-e_6^{-1})\e_{56,4}+ e_6)e_5^{-1}(-e_3^{-1})\e_{13,4}(-e_1^{-1})\cdot f_3\\
&= (-\e_{46,115} + \e_{6,12})f_1 + (\e_{34,155} - \e_{3,25})f_2,
\end{align*}
and also verify that $\wt\dd_\Gamma$ is indeed a differential:
\begin{align*}
\wt\dd_\Gamma^2 x &= -\e_{4,13}(-\e_{12,6}-\e_{23,5}+\e_{13,4}) + \e_{4,56}(-\e_{26,1}-\e_{25,3}+\e_{56,4}) = 0,\\
\wt\dd_\Gamma^2 y &= -e_1^{-2}(-\e_{12,6}-\e_{23,5}+\e_{13,4}) + (\e_{3,156} - \e_{3,246})(-\e_{26,1}-\e_{25,3}+\e_{56,4})\\
&\qquad - e_4^{-2}(\e_{45,6} - \e_{35,2} + \e_{34,1}) = 0,\\
\wt\dd_\Gamma^2 z &= (-\e_{46,115} + \e_{6,12})(-\e_{12,6}-\e_{23,5}+\e_{13,4})+ (\e_{34,155} - \e_{3,25})(-\e_{26,1}-\e_{25,3}+\e_{56,4})=0.
\end{align*}

\subsection{Blown-up 4-prism} Let us consider the graph $G$ depicted in Figure \ref{fig:1cube}, which consists of a modification of the 4-prism graph by addition of an interior edge, with the choice of garden. In this case we shall count unsigned binary sequences, the reader is invited to endow the garden with an orientation and compute the sign contributions.

\begin{figure}[h!]
\centering
  \includegraphics[scale=0.65]{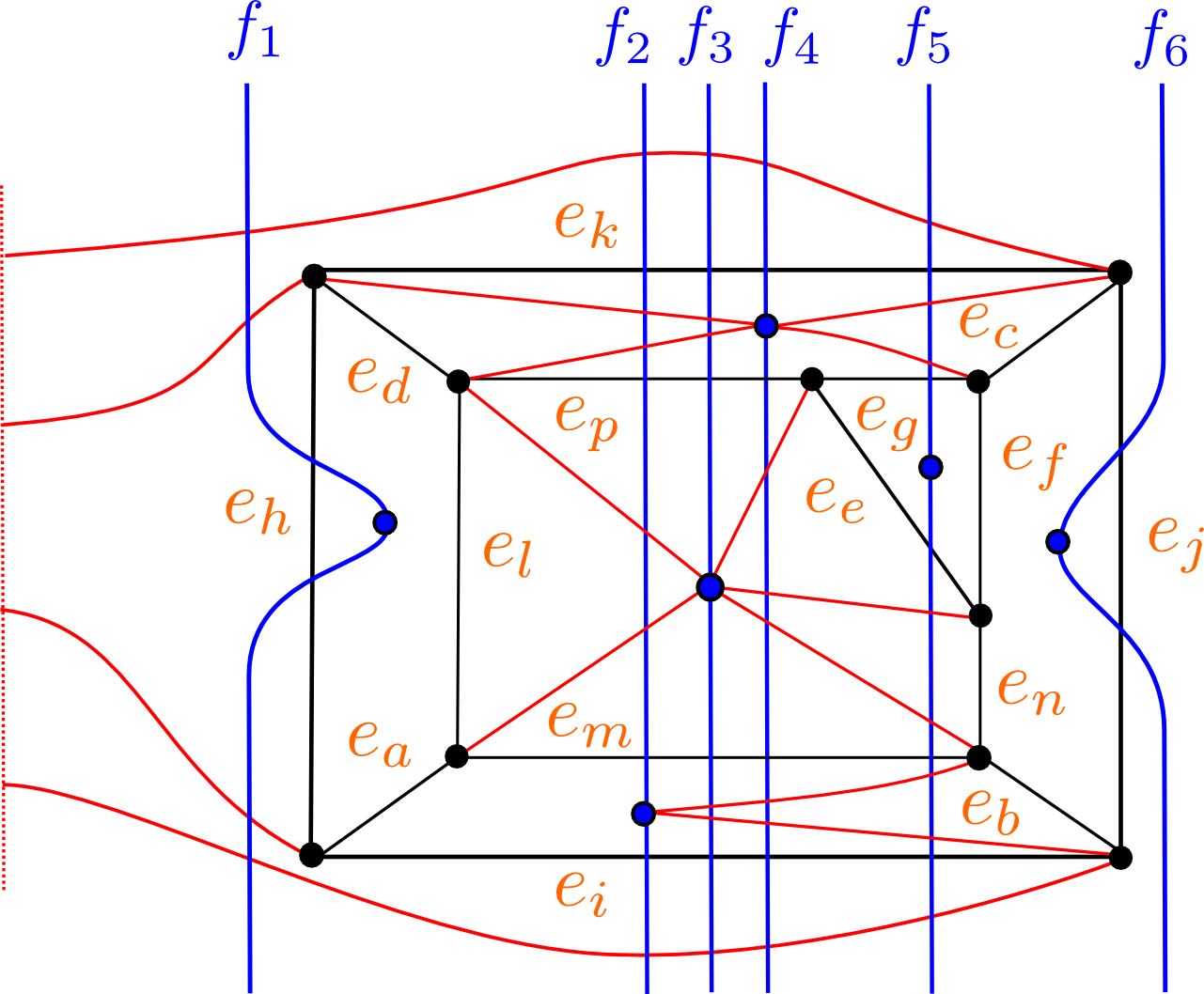}
    \caption{The modified 4-prism graph with a choice of garden.}\label{fig:1cube}
\end{figure}
\vspace{1cm}

The binary sequences that contribute to the degree-2 differential of the dg--algebra read as follows:
\begin{align*}
\wt\dd_\Gamma x&=[\e_{kj,c}(dy-\e_{0,jj}f_6)+\e_{hk,d}\e_{0,hh}f_1]+\e_{k,imp}(\e_{lm,a}+\e_{lp,d})f_2\\
&+\e_{k,imp}(\e_{ib,j}+\e_{mb,n})f_3+\e_{k,ieg}(\e_{m,0}+\e_{0,m}(\e_{ib,j}+\e_{mb,n})(\e_{mn,b}+\e_{en,f}))f_5,\\\\
\wt\dd_\Gamma y&= \e_{0,hh}f_1+\e_{0,ikm}(\e_{p,0}+\e_{0,p}(\e_{lm,a}+\e_{lp,d})(\e_{dp,l}+\e_{dk,h}))f_2\\
&+\e_{0,ikmp}(\e_{ib,j}+\e_{mb,n})(\e_{dp,l}+\e_{dk,h})f_3+\e_{0,ikp}(\e_{m,0}+\e_{0,m}(\e_{ib,j}+\e_{mb,n})(\e_{mn,b}+\e_{ne,f}+\e_{pe,g}))f_4\\
&+\e_{0,ikeg}(\e_{gc,f}+\e_{kc,j})(\e_{m,0}+\e_{0,m}(\e_{ib,j}+\e_{mb,n})(\e_{mn,b}+\e_{en,f}))f_5+\e_{0,jj}f_6,\\\\
\wt\dd_\Gamma z&=[\e_{ji,b}(dy+\e_{0,jj}\e_6)+\e_{0,hh}\e_{ih,a}\e_1]+\e_{i,mkp}(\e_{dp,l}+\e_{dk,h})\e_3\\
&+\e_{i,mkp}(\e_{nm,b}+\e_{en,f}+\e_{ep,g})\e_4+\e_{i,mkeg}(\e_{nm,b}+\e_{en,f})(\e_{gc,f}+\e_{ck,j})\e_5.\\
\end{align*}

Let us verify that in a field $\F$ of characteristic two, where the result is indepedent of sign contributions, the operator $\wt\dd_\Gamma$ squares to zero. Indeed, the square of the differential yields the following contributions:

\begin{align*}
\wt\dd_\Gamma^2x&= 2 + 2\e_{ b k,c i} + 2\e_{ a k,d i} + 2\e_{ f k,g j} + 2\e_{ b e k,
 g j m} + 2\e_{ a k l,d h m} + 2\e_{ b k l,d j m} + 2\e_{ f k m,
 c e i}\\
 &+ 2\e_{ b e k,g i n} + 2\e_{ b k l,d i n} + 2\e_{ f k m,
 g i n} + 2\e_{ f k n,c e j} + 2\e_{ b k n,c j m} + 2\e_{ b e k,
 f i p} + 2\e_{ k l,h p}\\
 &+ 2\e_{ b k l,a j p} + 4\e_{ k m,i p} + 2\e_{
  b k l m,a i n p} + 2\e_{ k n,j p} + 2\e_{ b e k n,f j m p}=0\\
\vspace{0.5cm}
\wt\dd_\Gamma^2y&= 2\e_{ a,h i} + 2\e_{ c f,g j^2} + 2\e_{ b,i j} + 2\e_{ d,h k} + 2\e_{ b g,
 f i k} + 2\e_{ c,j k} + 2\e_{ b d,a j k} \\
 &+ 2\e_{ b d e,f i k l} + 2\e_{ b c e,g j^2 m} + 2\e_{ b d e,g h j m} + 2\e_{ b c e,f j k m} + 2\e_{a l,h^2 m} + 2\e_{ b l,h j m} + 2\e_{ fm,e i j}\\
 &+ 4\e_{ g m,e i k} + 4\e_{ d m,i k l} + 2\e_{ b d e,g h i n} + 2\e_{ b c e,g i j n}+ 2\e_{ b c e,f i k n} + 2\e_{ b l,h i n} + 2\e_{ c f m,g i j n}\\
 &+ 2\e_{ c m,i k n} + 2\e_{ b d m,a i k n} + 2\e_{ f n, e j^2} + 2\e_{ g n,e j k} + 2\e_{ d n,j k l} + 2\e_{ b n,j^2 m} + 2\e_{ b g n,f j k m}\\
 &+ 2\e_{ b d e n,f j k l m} + 2\e_{ b d e,f h i p} + 2\e_{ b ce,f i j p} + 2\e_{ b c e g,f^2 i k p} + 2\e_{ d l,h^2 p} + 2\e_{ b d l,a h j p} + 4\e_{ d m,h i p}\\
 &+ 4\e_{ c m,i j p} + 4\e_{c g m,f i k p} + 2\e_{ b d l m,a h i n p} + 2\e_{ cn,j^2 p} + 2\e_{ d n,h j p} + 2\e_{ c g n,f j k p} + 2\e_{ b c e n,f j^2 m p}\\
 &+ 2\e_{ b d e n,f h j m p} + 2\e_{ b c e g n,f^2 j k m p} + 2\e_{ a p,i k l} + 2\e_{ a p,h k m} + 4\e_{ b p,j k m} + 2\e_{ b d e p,g j k l m} + 4\e_{ b p,i k n}\\
 &+ 2\e_{ b d e p,g i k l n}=0\\
\vspace{0.5cm}
\wt\dd_\Gamma^2z&=2 + 2\e_{ c f i,b g j} + 2\e_{ c i,b k} + 2\e_{ d i,a k} + 2\e_{ d e i,
 g h m} + 2\e_{ c e i,g j m} + 2\e_{ c e i,f k m}\\
&+ 2\e_{ i l,h m} + \e_{
 2 f i n,b e j} + 2\e_{ g i n,b e k} + 2\e_{ d i n,b k l} + 2\e_{ i n,
 j m} + 2\e_{ g i n,f k m} + 2\e_{ d e i n,f k l m}\\
&+ 2\e_{ d i l,
 a h p} + 2\e_{ d i n,b h p} + 2\e_{ c i n,b j p} + 2\e_{ c g i n,
 b f k p} + 2\e_{ d e i n,f h m p} + 2\e_{ c e i n,f j m p} + \e_{
 2 c e g i n,f^2 k m p}\\
&+ 2\e_{ i p,k m} + 2\e_{ d e i p,g k l m}=0
\end{align*}

\section{Invariant and T-versal dg-algebras of binary sequences} \label{sec:real DGA}

Let $G$ be a bridgeless cubic planar graph, in this section we construct a dg-algebra $(\A_G, \dd_\Gamma)$ from the dg-algebra $(\wt \A_G, \wt \dd)$ of binary sequences. For reasons that will become apparent, we refer to $(\A_G, \dd_\Gamma)$ as the invariant dg-algebra of binary sequences.

Extend the algebra $\wt\A_G$ by adding a central invertible element $t$, to obtain the algebra $\wt\A_G[t^{\pm 1}]$, which is defined as a dg-algebra by assigning $t$ grading $0$ and defining $\wt\dd_\Gamma t = 0$. Given a face $f \in F$, consider the algebra map
$$\p_f: \wt \A_G[t^{\pm 1}] \lr \wt \A_G[t^{\pm 1}],\qquad \p_f(e_i)=
\begin{cases}
t e_i &\mbox{if }e_i \sse f,\\
e_i &\mbox{if }e_i\not\sse f,
\end{cases}\qquad
\p_f(f_j)=
\begin{cases}
t^2 f_j &\mbox{if }f_j = f,\\
f_j &\mbox{if }f_j\neq f,
\end{cases}
$$

and $\p_f(x) = x$, $\p_f(y) = y$, $\p_f(z) = z$. In addition we define
$$\p_0: \wt \A_G[t^{\pm 1}] \to \wt \A_G[t^{\pm 1}],\quad \p_0(e_i) = t e_i,\quad \p_0(f_j) = t f_j,\quad \p_0(x) = x,\quad\p_0(y) = t^{-1}y,\quad \p_0(z) = z,$$
for all $e_i \in E$ and $f_j \in F$. It is readily seen that these maps commute and therefore define an action of the lattice $\Z^{g+3}$ on $\wt \A_G[t^{\pm 1}]$ by graded algebra isomorphisms. The following proposition shows that this actually defines an action on $(\wt \A_G[t^{\pm 1}], \wt \dd_\Gamma)$ by dg-algebra isomorphisms.

\begin{prop}\label{prop:action}
Let $G$ be a planar cubic bridgeless graph equipped with a garden $\Gamma$. Then
$$\wt \dd_\Gamma \circ \p_0 = \p_0 \circ \wt \dd_\Gamma,\qquad \wt \dd_\Gamma \circ \p_f = \p_f \circ \wt \dd_\Gamma,\quad \forall f \in F.$$
\end{prop}

\begin{proof}
It suffices to prove this for each face $f_j \in F$ and the three generators $x, y,$ and $z$.

Consider a face $f_j \in F$, then it is immediate that $(\p_0 \circ \wt \dd_\Gamma) (f_j) = (\wt \dd_\Gamma \circ \p_0)(f_j)$ since $\wt \dd_\Gamma f_j$ is a homogeneous Laurent polynomial in $E$ of degree $1$, and $\p_0$ multiplies $f_j$ and all $e_j \in E$ by $t$.

Suppose $f_j \neq f$, then we need to show that $\p_f(\wt \dd_\Gamma f_j) = \wt \dd_\Gamma f_j$. The Laurent polynomial $\wt \dd_\Gamma f_j$ is a sum over the vertices $v \in f_j$ of terms $\pm e_me_ne^{-1}_k$. If a given vertex $v$ is not contained in $f$, then none of the edges $e_m$, $e_n$, $e_k$ are contained in $f$, and therefore $\p_f$ acts trivially on this term. If $v$ is contained in $f$, then the two faces $f$ and $f_j$ share exactly one edge from $e_m, e_n$, and $e_k$, since $G$ is bridgeless. By definition the edge $e_k$ is not contained in the face $f_j$ and in consequence $e_k \sse f$. Suppose that $e_m \sse f \cap f_j$, then the automorphism $\p_f$ multiplies $e_m$ and $e_k$ by $t$ and leaves $e_n$ fixed, which implies that the term $\pm e_me_ne_k^{-1}$ is left fixed by $\p_f$.

Suppose that $f = f_j$, then we must prove $\p_f(\wt \dd_\Gamma f) = t^2 \wt \dd_\Gamma f$. The Laurent polynomial $\wt \dd_\Gamma f$ is a sum of terms $\pm e_m e_n e_k^{-1}$ with $e_m, e_n \sse f$, whereas $e_k$ is not contained in $f$ and therefore $\p_f$ multiplies all terms $\pm e_m e_n e_k^{-1}$ by $t^2$. This completes the proof that $\wt \dd_\Gamma$ commutes with $\p_0$ and $\p_f$ on $(\wt \A_G)_1$.

We now consider $(\wt\A_G)_2$. Let $B \in \mathcal{B}_{a,b}(f_j)$ a binary sequence along the tine $\gamma_{f_j}$. By definition of the differential $\wt \dd_\Gamma$ acting on $x, y$ and $z$, the following four equalities imply the desired statement: 

\begin{equation}\label{eqn:action}
\p_f(H(B)) = H(B) \text{ whenever } f_j \neq f,\qquad \p_f(H(B)) = t^{-2} H(B) \text{ if } f_j = f,\end{equation}
$$\p_0(H(B)) = t^{-1} H(B) \text{ if } a=b,\qquad \p_0(H(B)) = t^{-2} H(B) \text{ if } a=1, b=0.$$

We now establish each of these claims. For the first equality, note that $H(\tau) = \pm e_n e_m e_k ^{-1}$ is fixed by $\p_f$ unless $\tau$ is a thread in the web associated to the face $f$, in which case $\p_f(H(\tau)) = t^2 H(\tau)$. This is for the same reason as above: either the thread $\tau$ is completely disjoint from $f$ -- in which case $\p_f$ acts trivially -- or else $\p_f$ multiplies two of the edges by the constant $t$. The fact that these terms might cancel or not is determined by whether $\tau$ is a thread contained in $f$. Therefore the automorphism $\p_f$ only affects $H(B)$ by acting on those factors which arise from the segments of $\gamma_{f_j}$ passing through the face $f$. There are three possibilities each time this occurs:

\begin{itemize}
 \item[1.] $B$ switches from $0$ to $1$ when it enters $f$ and switches from $1$ back to $0$ when it exits,
 \item[2.] $B$ switches from $1$ to $0$ when it enters $f$ and switches from $0$ back to $1$ when it exits,
 \item[3.] $B$ switches from $1$ to $0$ when it enters $f$, switches from $0$ to $1$ at a thread $\tau$ in the web of $f$, and switches from $1$ to $0$ again when it exits $f$.
\end{itemize}

A switch at an edge $e_i$ contributes a factor $\pm e_i^{\pm 1}$ to $H(B)$, where the sign of the exponent being determined by whether the switch is $1$ to $0$ or $0$ to $1$, and therefore the automorphism $\p_f$ acts invariantly in the first two cases. In the third case, the exponent is negative for both the entrance and the exit and thus the $t^{-2}$ factor that comes from applying the automorphism $\p_f$ to these factors exactly cancels the $t^2$ factor coming from applying $\p_f$ to $H(\tau)$. This proves the first equality in Equation \ref{eqn:action}.

Consider the second equality, where $f_j = f$. The argument proceeds exactly as before, except that there is a further possibility when $\gamma_f$ enters $f$:
\begin{itemize}
 \item[4.] $B$ switches from $1$ to $0$ when it enters $f$, it switches from $0$ to $1$ at $c_f$, and switches from $1$ back to $0$ when it exits $f$.
\end{itemize}

This fourth case must happen exactly once along the binary sequence $B$, and since the switch at the center $c_f$ contributes no factors to $H(B)$, once we apply the automorphism $\p_f$ we are left with a factor of $t^{-2}$, as required.

For the remaining two equalities in Equation \ref{eqn:action} we use that $\p_0$ multiplies all faces and edges by $t$. Every time that a binary sequence $B$ switches from $1$ to $0$, which can only happen when the tine $\gamma_{f_j}$ crosses an edge, the automorphism $\p_0$ introduces an additional factor of $t^{-1}$ to $H(B)$. In the case where $B$ switches from $0$ to $1$, it is either at an edge, a thread or at the center $c_{f_j}$. In the former two cases $\p_0$ introduces a factor of $t$ and in the latter case, which occurs exactly once, acts trivially. Therefore the exponent of $t$ which comes from applying the automorpshim $\p_0$ to $H(B)$ is the total count of switches along the binary sequence $B$, which is $\p_0(H(B)) = t^{b-a-1} H(B)$.\end{proof}

\begin{definition}\label{def:A}
Let $(G,\Gamma)$ be a cubic planar bridgeless graph endowed with a garden. We define $\A_G \sse \wt \A_G[t^{\pm 1}]$ to be the subalgebra of elements which are contained in $\wt \A_G$ and which are fixed by the action of $\Z^{g+3}$ via dg-isomorphisms. The differential $\wt \dd_\Gamma$ restricts to a differential $\dd_\Gamma: \A_G \lr \A_G$, and the dg-isomorphism type of the dg-algebra $(\A_G, \dd_\Gamma)$ is referred to as the {\bf invariant dg-algebra of binary sequences}.\hfill$\Box$
\end{definition}

\begin{remark}\label{rmk:bridgeless}
The algebra $(\A_G, \dd_\Gamma)$ can be defined for a graph $G$ containing bridges, in which case the definition of the $\Z^{g+3}$ action must be slightly generalized; we define $\p_f(e) = t^2 e$ if $e$ is the edge of $f$ from both sides. We refrain from discussing this situation in order to avoid even more cases in the proof above, and also to avoid the exceptional cases when discussing the upcoming $T$-versal dg-algebras.

In addition, note that the dg-algebra $(\wt\A_G, \wt \dd)$ is acyclic if $G$ contains a bridge. Indeed, in that case $\wt\dd h=1$ for some $h\in\wt\A_G$ and therefore $1=0$ in the cohomology $H_*(\wt\A_G, \wt\dd)$, which implies its vanishing. This is also reflected by the fact that the dual graph $G^*$ contains a loop and therefore it does not admit any colorings, see Appendix \ref{app:kevin}.\hfill$\Box$
\end{remark}

Let us prove that the dg-isomorphism type of $(\A_G, \dd_\Gamma)$ does not depend on the choice of garden $\Gamma$.

\begin{thm}\label{thm:A inv}
Let $(G,\Gamma)$ be a decorated cubic planar bridgeless graph. The invariant dg-algebra of binary sequences $(\A_G, \dd_\Gamma)$ does not depend on the choice of garden $\Gamma$.
\end{thm}

\begin{proof}
It suffices to prove that the isomorphisms constructed in Section \ref{sec:invar} commute with the action of $\Z^{g+3}$. In the course of the section there are three moves which induce a non-identity isomorphism: changing the rake by a tine switching, changing the threads at infinity of the web and changing the orientation of an edge.

The isomorphism induced by changing threads at infinity is given in Equation \ref{eqn:anchor iso}, which is readily seen to commute with $\Z^{g+3}$: the action of $\Z^{g+3}$ is trivial on $x$ and $z$, and thus it suffices to show that it is also trivial on $e_me_ne_k^{-1}y$. For the automorphism $\p_0$ this is satified since all edges are multiplied by $t$ whereas $y$ is multiplied by $t^{-1}$, for the automorphism $\p_f$ it is enough to notice that whenever $f$ contains $e_m$ it must also contain $e_k$ but never $e_n$, since by definition these three edges occur at an exterior vertex.

In the case of a tine switching, the induced isomorphism is given in Equation \ref{eqn:tine iso}, and we claim that this isomorphism commutes with the action of $\Z^{g+3}$. This follows from the next four claims for any binary sequence $B$ along a tine $\gamma_{m,n}$:
$$\p_f(H(B)) = H(B) \text{ whenever } f_m \neq f \neq f_n,\qquad\p_f(H(B)) = t^{-2} H(B) \text{ if } f = f_m \text{ or } f = f_n$$
$$\p_0(H(B)) = t^{-2} H(B) \text{ if } a=b,\qquad \p_0(H(B)) = t^{-3} H(B) \text{ if } a=1, b=0.$$
This is similar to Equation \ref{eqn:action} and proof is identical except that we consider $\gamma_{m,n}$ instead of $\gamma_{f_j}$. 

Finally, if we change the orientation of an edge $e_i$, the induced isomorphism sends $e_i$ to $-e_i$ and fixes the remaining elements, and this isomorphism action commutes with the action of $\Z^{g+3}$.\end{proof}

Being an algebra of invariant elements, it is difficult to compute $(\A_G,\dd_\Gamma)$ using Definition \ref{def:A}. In order to compute $(\A_G, \dd_\Gamma)$ in practice, we construct a slice for our quotient, in which we can compute more directly. Let us define a {\bf basis} of $G$ to be a choice of a vertex $v_0 \in V$ together with a tree $T \sse G$ which spans the vertex set $V \sm \{v_0\}$; notice that any basis consists of $2g$ edges.

\begin{definition}\label{def:tree A}
Let $(G,\Gamma)$ be a decorated planar cubic bridgeless graph with a basis $T=\{e_1, \ldots, e_{2g}\}$ and $\Lambda_T=\F[e_1, e_1^{-1}, \ldots, e_{2g}, e_{2g}^{-1}]$ the algebra of Laurent polynomials in these edges. Consider the map $Q: \Lambda_G \lr \Lambda_T$ which sends the edges in $E \sm T$ to the unit $1\in\Lambda_T$. We define the $\Lambda_T$-algebra

$$\A_T = \langle x, y, z, f_1, \ldots, f_{g+2} \rangle_{\Lambda_T}$$

and the induced map $Q: \wt \A_G \to \A_T$. Since $Q$ is surjective and $\wt \dd_\Gamma$ vanishes on the kernel of $Q$, this uniquely defines a map
$$\dd_T: \A_T \to \A_T,\qquad Q \circ \wt \dd_\Gamma = \dd_T \circ Q.$$
The pair $(\A_T, \dd_T)$ is said to be the {\bf $T$-versal dg-algebra of binary sequences} of $G$.\hfill$\Box$
\end{definition}

\begin{thm}\label{thm:basis inv}
Let $(G,\Gamma,T)$ be a decorated planar cubic bridgeless graph. Then the $T$-versal dg-algebra $(\A_T, \dd_T)$ is isomorphic to $(\A_G, \dd_\Gamma)$. In particular, the dg-isomorphism type of $(\A_T, \dd_T)$ is independent of the basis $T$.
\end{thm}

We first prove a lemma:

\begin{lemma}\label{lem:slice}
Let $(G,\Gamma,T)$ be a decorated planar cubic bridgeless graph. For any map $\beta: E \sm T \lr \Z$, there is a unique $\p \in \Z^{g+3}$, such that for every $e_i \in E \sm T$, $\p(e_i) = t^{\beta(e_i)} e_i$.
\end{lemma}

\begin{proof}
Let us consider the automorphism $\p = \p_0^{\lambda_0} \circ \p_{f_1}^{\lambda_1} \circ \ldots \circ \p_{f_{g+2}}^{\lambda_{g+2}}$. Then, for each $e_i \in E \sm T$, we can write
$$\p(e_i) = t^{\lambda_0+\lambda_{f_m}+\lambda_{f_n}} e_i,$$
where $f_m$ and $f_n$ are the two faces containing $e_i$. In case $e_i$ is an exterior edge, then $\p(e_i) = t^{\lambda_0+\lambda_{f_m}} e_i$.

Therefore the system of equations $\{\p(e_i) = t^{\beta(e_i)}e_i\}$ for $e_i \in E \sm T$ is a set of $g+3$ equations in $g+3$ variables ($\lambda_0, \ldots \lambda_{g+2}$), which we claim has a unique solution for any $\beta$. To write this linearly, let $a_{ij}$ equal either $1$ or $0$ according to whether $e_i$ is adjacent to $f_j$, where $1 \leq i \leq g+3$ indexes all edges in $E \sm T$ and $1 \leq j \leq g+2$ indexes all faces. In addition, declare $a_{i0} = 1$ for all $i$ and consider $g+3 \x g+3$ matrix $A=(a_{ij})$. It suffices to show that $\det A = \pm 1$, since the equation to solve is

$$
\left(\rule{0cm}{1cm} \qquad A \qquad \right)
\left( \begin{array}{cc}
\lambda_0  \\
\lambda_1  \\
\vdots  \\
\lambda_{g+2}
\end{array} \right) = \left( \begin{array}{cc}
\beta(e_1)  \\
\beta(e_2) \\
\vdots  \\
\beta(e_{g+3})
\end{array} \right).
$$

Let $G^*_T$ be the partial dual of $G$ defined according to the condition that the vertices of $G^*_T$ are the faces of $G$, including the exterior face $f_0$, and the edges of $G^*_T$ are the edges of $G$ which are not in $T$. Since $T$ is a tree which spans all vertices of $G$ except for one, $G^*_T$ is a connected graph which has a unique embedded cycle, which is a triangle. The adjacency matrix of $G^*_T$ is nearly the same as the matrix $A$, the only difference is that we remove the column corresponding to $f_0$, and replace it with a column of all $1$s. In particular, if $G^*_T$ has a vertex different from $f_0$ of valence $1$, the corresponding column of $A$ has all $0$s except for a single $1$. For the purpose of computing $\det A$, we can then expand along this column without changing the determinant up to sign, which has the effect of removing this vertex from $G^*_T$ as well as the edge adjacent to it.

Therefore, we can assume without losing generality that $G^*_T$ has no vertices of valence $1$ except possibly $f_0$, that is, a triangle with a chain connecting to $f_0$. The resulting matrix is then
$$
A= \left( \begin{array}{cccccccc}
1 & 1 & 1 & 0 & 0 & \ldots & 0 & 0  \\
1 & 1 & 0 & 1 & 0 & \ldots & 0 & 0  \\
1 & 0 & 1 & 1 & 0 & \ldots & 0 & 0  \\
1 & 0 & 0 & 1 & 1 & \ldots & 0 & 0 \\
1 & 0 & 0 & 0 & 1 & \ldots & 0 & 0 \\
\vdots & \vdots & \vdots & \vdots & \vdots & \ddots & \vdots & \vdots \\
1 & 0 & 0 & 0 & 0 & \ldots & 1 & 1 \\
1 & 0 & 0 & 0 & 0 & \ldots & 0 & 1
\end{array} \right)$$
which is readily seen to have determinant $\pm 1$.\end{proof}

\begin{proof}[Proof of Theorem \ref{thm:basis inv}:]
As before we write an arbitrary $\p \in \Z^{g+3}$ as $\p = \p_0^{\lambda_0} \circ \p_{f_1}^{\lambda_1} \circ \ldots \circ \p_{f_{g+2}}^{\lambda_{g+2}}$. We also write $\p(e_i) = t^{\beta(e_i)}e_i$, where we adopt the convention that $E \sm T = \{e_1, \ldots, e_{g+3}\}$ and $T = \{e_{g+4}, \ldots, e_{3g+3}\}$. As before, we write an adjacency matrix satisfying

$$
\left(\begin{array}{cc}
A\\
\,
A_T
\end{array}\right) \left( \begin{array}{cc}
\lambda_0  \\
\lambda_1  \\
\vdots  \\
\lambda_{g+2}
\end{array} \right) = \left( \begin{array}{cc}
\beta(e_1)  \\
\vdots  \\
\beta(e_{3g+3})
\end{array} \right).
$$

Here $A$ is the same $g+3 \x g+3$ matrix appearing in the previous proof, and $A_T$ is a $2g \x g+3$ matrix defining how $\p$ acts on the edges in $T$. In Lemma \ref{lem:slice} it is shown that $A$ is invertable. Consider the elements $\wh e_i = e_i\prod_{r=1}^{g+3}e_r^{b_{ir}}\in \Lambda_G$, which are defined by

$$
\left(\rule{0cm}{0.5cm}\quad b_{ir} \quad \right) = - \left(\begin{array}{cc}
A\\
\,
A_T
\end{array}\right) \left(\rule{0cm}{0.5cm} \quad A^{-1} \quad \right) = - \left(\begin{array}{cc}
\op{Id}\\
\,
A_TA^{-1}
\end{array}\right).
$$

In particular we have $\wh e_i = 1$ for $1 \leq i \leq g+3$. We claim $(\A_G)_0$ is equal to the algebra of Laurent polynomials in  $\{\wh e_{g+4}, \ldots, \wh e_{3g+3}\}$, which is equivalent to the claim that $Q_T|_{(\A_G)_0} :(\A_G)_0  \lr (\A_T)_0$ is an isomorphism. First, we note that $\wh e_i \in (\A_G)_0$, since

$$
\p(\wh e_i) = t^{\beta(e_i) + \sum_r b_{ir} \beta(e_r)} \wh e_i,
$$
and
$$
\left(\begin{array}{cc}
\sum_r b_{1r}\beta(e_r)\\
\vdots\\
\sum_r b_{(3g+3)r}\beta(e_r)
\end{array}\right) = - \left(\begin{array}{cc}
A\\
\,
A_T
\end{array}\right) \left(\rule{0cm}{0.5cm} \quad A^{-1} \quad \right) \left(\rule{0cm}{0.5cm} \quad A \quad \right) \left(\begin{array}{cc}
\lambda_0\\
\vdots \\
\lambda_{g+2}
\end{array}\right) = -\left(\begin{array}{cc}
\beta(e_1)\\
\vdots \\
\beta(e_{3g+3})
\end{array}\right).
$$

To see that every element of $(\A_G)_0$ is a Laurent polynomial in  $\{\wh e_{g+4}, \ldots, \wh e_{3g+3}\}$, first note that $(\A_G)_0$ consists of linear combinations of monomials in $(\A_G)_0$, since $\Z^{g+3}$ acts multiplicatively. We identify the multiplicative group of monomials in $\wt \A_G[t^{\pm 1}]$ with the additive group $\Z^{3g+4}$ by identifying $t^{k_0} e_1^{k_1}\ldots e_{3g+3}^{k_{3g+3}}$ with $(k_0, k_1, \ldots, k_{3g+3})$. The group $\Z^{g+3}$ is acting linearly on this space, where the element $\p$ acts by
$$
\p = \left( \begin{array}{ccccc}
1 & \beta(e_1) & \ldots & \beta(e_{3g+2}) & \beta(e_{3g+3})\\
0 & 1 & \ldots & 0 & 0\\
\vdots & \vdots & \ddots & \vdots & \vdots \\
0 & 0 & \ldots & 1 & 0 \\
0 & 0 & \ldots & 0 & 1
\end{array} \right).$$
We immediately see then that the monomials in $(\A_G)_0[t^{\pm 1}]$ are identified with the orthocomplement of the linear subspace $S = \left\lbrace (0, \beta(e_1), \ldots, \beta(e_{3g+3}): (\lambda_0, \ldots, \lambda_{g+2}) \in \Z^{g+3}\right\rbrace$. Since the matrix $\left(\begin{array}{cc}
A\\
\,
A_T
\end{array}\right)$ is full rank, we see that $\op{rank}(S) = g+3$, so its orthocomplement has rank $2g+1$. Furthermore every monomial in $\{t, \wh e_{g+4}, \ldots, \wh e_{3g+3}\}$ is primitive in $(\wt \A_G)_0[t^{\pm 1}]$, so $(\A_G)_0[t^{\pm 1}] / \{t, \wh e_{g+4}, \ldots, \wh e_{3g+3}\}$ cannot have torsion. This completes the proof of $(\A_G)_0 \cong (\A_T)_0$.

For grading $1$, we define the $g+2 \x g+3$ matrix $A_F$ by
$$
A_F = \left( \begin{array}{ccccc}
1 & 2 & 0 & \ldots & 0\\
1 & 0 & 2 & \ldots & 0\\
\vdots & \vdots & \vdots & \ddots & \vdots \\
1 & 0 & 0 & \ldots & 2\\
\end{array} \right)$$
which describes how $\Z^{g+3}$ acts on faces. Let $\mu_{jr}$ be the terms of $-A_F A^{-1}$, and define $\wh f_j = f_j \prod_{r=1}^{g+3}e_r^{\mu_{jr}}$. It follows that $\p(\wh f_j) = \wh f_j $ for any $\p \in \Z^{g+3}$. This immediately implies that $\{\wh f_j\}$ is a basis for $(\A_G)_1$ over $(\A_G) _0$. Finally, a basis for $(\A_G)_2$ consists of $\wh x = x$, $\wh z = z$, and $\wh y = y \prod_{r=1}^{g+3}e_r^{\eta_r}$, where $\{\eta_r\}$ forms the first row of $A^{-1}$.
\end{proof}


\appendix
\maketitle

\section{Augmentations are Colorings}
\begin{center}
{\footnotesize KEVIN SACKEL}\daggerfootnote{Massachusetts Institute of Technology, Department of Mathematics, 77 Massachusetts Avenue Cambridge, MA 02139, USA.} \label{app:kevin}\\
\end{center}

\subsection{Introduction}
In this appendix, we prove Theorem \ref{thm:augs_colors}, Theorem \ref{thm:main kevin} in the body of the paper. The motivation for this theorem is the expected correspondence between augmentations of the Legendrian dg-algebra and the moduli space of rank-1 constructible sheaves with singular support along the Legendrian. This correspondence was originally conjectured in \cite{STZ} and proved for Legendrian knots in \cite{NRSSZ}. With this philosophy in mind, the cubic planar graphs considered here correspond to the Legendrian surfaces considered in \cite{TZ}, face colorings up to reparametrization correspond to sheaves, and the invariant dg-algebra constructed in the main body of the text is expected to be the Legendrian dg-algebra in the sense of symplectic field theory \cite{EES, EGH}.

For the remainder of this appendix, the reader may disregard the motivation entirely and view the main theorem as purely combinatorial in nature. Let $\F$ be an underlying ground field and all graphs $G$ considered here will be cubic planar, and for simplicity, bridgeless. As noted in Remark \ref{rmk:bridges}, this last assumption can be dropped in the statement of the theorem.

\begin{definition} An \textbf{augmentation} of a dg-algebra $(\A,\dd)$ is a morphism dg-algebra map
$$\epsilon \colon (\A,\dd) \rightarrow (\F,0),$$
where $(\F,0)$ is the dg-algebra given by $\F$ concentrated in degree zero.
\end{definition}

Note that in order to define an augmentation, it suffices to specify $\epsilon$ on some set of generators $\mathcal{G} \subset \A_0$, which we assume is finite. Hence, one can identify $\mathrm{Aug}(\A,\dd;\F)$ as a subspace of the affine space $\F^{\mathcal{G}}$. Suppose now we are in a particularly nice situation, where $\A_i = 0$ for $i < 0$. Then an element of $\F^{\mathcal{G}}$ defines an augmentation precisely when it satisfies:
\begin{itemize}
	\item[-] The algebraic relations imposed by the equations defining $\A_0$ as an algebra.
	\item[-] The algebraic relations imposed by the equations coming from $\dd\A_1$. In particular, if $\A_1$ is finitely generated as an $\A_0$-module, then one can specify equivalent data by a finite list of relations.
\end{itemize}
If furthermore $\A_0$ is commutative, then these relations are all just polynomials in the generators of $\mathcal{G}$. Hence, $\mathrm{Aug}(\A,\dd;\F) \subset \F^{\mathcal{G}}$ forms a variety given by the corresponding ideal.

\begin{definition}
In the set-up just described, we shall refer to $\mathrm{Aug}(\A,\dd;\F) \subset \F^{\mathcal{G}}$ as the \textbf{augmentation variety}.
\end{definition}

\begin{remark}
Under these nice hypotheses, we are thinking of $\mathrm{Aug}(\A,\dd;\F)$ as a variety, and so in cases when $\F \neq \overline{\F}$, there may be more structure than just the underlying set of augmentations. In addition, note that we obtain an isomorphic variety if we choose a different set of generators $\mathcal{G}' \subset \mathcal{A}_0$, since we have rational morphisms $\F^{\mathcal{G}} \rightarrow \F^{\mathcal{G}'}$ and $\F^{\mathcal{G}'} \rightarrow \F^{\mathcal{G}}$ given by writing the generators of one set in terms of the other. The composition $\F^{\mathcal{G}} \rightarrow \F^{\mathcal{G}'} \rightarrow \F^{\mathcal{G}}$ is the identity on the variety given by the algebraic relations coming from $\A_0$, a fortiori on the augmentation variety. So these varieties are indeed isomorphic.
\end{remark}

One would like to understand the augmentations of the dg-algebras described in the main body of this paper. In the case of the dg-algebra of binary sequences for the graph $G$, one can take as a generating set $\mathcal{G} := \{e_1,e_1^{-1},\ldots,e_{3g+3},e_{3g+3}^{-1}\}$, where we have once and for all fixed an enumeration on and orientation of the edges. The algebraic relations (aside from commutativity) imposed by $\A_0$ are simply that $e_ie_i^{-1} = 1$, and so one could also think about our augmentation variety as living in $(\F^*)^{3g+3}$ instead of $\F^{6g+6}$.

\begin{definition} The variety $\widetilde{X_G(\F)} := \mathrm{Aug}(\widetilde{\A}_G,\widetilde{\dd}_G;\F)$ is given as
$$\left\{(\lambda_1,\ldots,\lambda_{3g+3}) \in (\F^*)^{3g+3}; \wt\dd f_j = \sum_{v \in f_j} H(\tau_{f_j}(v)) = 0 \text{ for all }j, \mathrm{where~}e_i=\lambda_i\right\}$$
and is called the \textbf{augmentation variety of binary sequences of $G$}.
\end{definition}

\begin{definition} For a basis $(v_0,T)$ of $G$, the variety $X_{T}(\F) := \mathrm{Aug}(\A_T,\dd_T;\F)$ is given as
$$\left\{(\lambda_1,\ldots,\lambda_{3g+3}) \in \widetilde{X_G(\F)}; \lambda_i = 1 \mathrm{~when~} e_i \notin T\right\}$$
and is called the \textbf{$T$-versal augmentation variety of $G$}.
\end{definition}

\begin{definition}
The variety $X_G(\F) := \mathrm{Aug}(\A_G,\dd;\F)$ is called the \textbf{invariant augmentation variety of $G$}.
\end{definition}

The invariant dg-algebra $(\A_G,\dd)$ embeds in $(\wt{A}_G,\wt{\dd}_G)$, so this induces a map $\wt{X_G(\F)} \rightarrow X_G(\F)$ on augmentation varieties. In addition, since $(\A_G,\dd) \simeq (\A_T,\dd_T)$, we have a canonical isomorphism $X_G(\F) \simeq X_T(\F)$. In the following, we shall need to understand the projection map given by the composition $\wt{X_G(\F)} \rightarrow X_G(\F) \xrightarrow{\sim} X_T(\F)$. The proof of Theorem \ref{thm:basis inv} implies that this morphism is described by sending an element $\vec{\lambda} \in \wt{X_G(\F)}$ to the element $\vec{\lambda}' \in X_T(\F)$ with, in the same notation as that proof,
$$\lambda'_k := \lambda_k \cdot \prod_{j=1}^{g+3} \lambda_j^{b_{kj}}$$
for each $k$ (and in particular $\lambda'_k = 1$ for edges not in $T$, i.e. for $k=1,\ldots,g+3$).

On the other hand, one can consider colorings of the faces of $G$ by elements of $\F\PP^1$. We think of this dually as a vertex coloring of the dual graph $G^*$.

\begin{definition}
The \textbf{pre-chromatic variety of $G^*$} is
$$\widetilde{\chi_{G^*}(\F\PP^1)} = \left\lbrace (\kappa_0, \ldots, \kappa_{g+2}) \in (\F\PP^1)^{g+3}; \kappa_i \neq \kappa_j \text{ if } f_i \text{ is adjacent to } f_j\right\rbrace$$
\end{definition}

On this variety, there is a natural action of $\op{PSL}_2(\F)$, given by simultaneous action on each $\F\PP^1$-coordinate.

\begin{definition}
The \textbf{chromatic variety of $G^*$} is
$$\chi_{G^*}(\F\PP^1) = \widetilde{\chi_{G^*}(\F\PP^1)} / \op{PSL}_2(\F),$$
\end{definition}

This is a smooth quotient, and the chromatic variety is isomorphic to the corresponding slice of $\wt{\chi_{G^*}(\F\PP^1)}$ in which we fix $(\kappa_i,\kappa_j,\kappa_k) = (0,1,\infty)$ around a given vertex.

\begin{thm} \label{thm:augs_colors}
For any field $\F$, there is a canonical isomorphism $\Phi \colon X_G(\F) \rightarrow \chi_{G^*}(\F\PP^1)$ as varieties over $\F$ (not just as sets of points).
\end{thm}

\begin{remark} \label{rmk:bridges} For the case when $G$ has a bridge, the theorem is still valid, since Remark \ref{rmk:bridgeless} implies $\wt{X_G(\F)}$ is empty, while $\chi_{G^*}(\F\PP^1)$ is also clearly empty since some face is adjacent to itself.
\end{remark}

\begin{remark}
In the case of $\F = \F_3$, rational points of the augmentation variety correspond to four-colorings of the faces of $G$ (up to reparametrization), and so the famous Four Color Theorem (which can easily be reduced to studying bridgeless cubic planar graphs) is equivalent to the existence of a rational point on $\wt{X_G(\F_3)}$ when $G$ is bridgeless. If we have such a rational point, then it yields a vertex coloring by elements $\pm 1$ such that the sum around any face is $0$ mod $3$. In fact, this stronger statement is also equivalent to the Four Color Theorem, and was well known prior to the famous and controversial computer-based proof of \cite{AH,AHK}, in fact as far back as 1898 \cite{He}. Hence, one expects no new purely combinatorial insight into the Four Color Theorem.
\end{remark}

The proof of Theorem \ref{thm:augs_colors} occupies the following two subsections. In Subsection \ref{subsect:examples}, we provide two examples of our isomorphism, and finally, in Subsection \ref{subsect:numtheory}, we describe an alternate viewpoint on the viewpoint which leads to a combinatorial/number-theoretic identity, Corollary \ref{cor:mu_col}.

{\bf Acknowledgements}: K.~Sackel is grateful Bjorn Poonen for pointing towards EGA. This work is supported by an NSF Graduate Research Fellowship.

\subsection{Construction of the morphism}

We shall proceed by constructing a morphism of the form $\widetilde{\Phi} \colon \widetilde{X_G(\F)} \rightarrow \chi_{G^*}(\F\PP^1)$ which restricts to the desired morphism $\Phi \colon X_T(\F) \rightarrow \chi_{G^*}(\F\PP^1)$. We will then show that $\Phi$, considered instead as a morphism $X_G(\F) \rightarrow \chi_{G^*}(\F\PP^1)$, does not depend on the choice of $T$.

Consider an element $(\lambda_1,\ldots,\lambda_{3g+3}) \in \widetilde{X_G(\F)}$. For each choice of face-vertex adjacency $(f,v)$, this defines an element of $\F^*$ given by $(-1)^{r_v} \lambda_m\lambda_n\lambda_k^{-1}$, the image of $H(\tau_f(v))$ under the augmentation, which we shall also denote by abuse of notation as $H(\tau_f(v))$. By definition, the sum of these values around the vertices of a given face is $0$. The relevant quantity with respect to defining the coloring is not actually the precise value of each of these quantities, but the ratios of the quantities to each other around a face, which will be the cross ratio of the labellings in the coloring around an edge, as we now describe.

Suppose we consider two vertices $v,w$ connected by an edge $p$, with the other edges around $v$ given by $q,r$ and around $w$ as $s,t$, such that $q,r,s,t$ are labelled clockwise. Label the faces clockwise by $A,B,C,D$ such that $q$ separates $A$ and $B$ (and so $r$ separates $B$ and $C$ and so on). See Figure \ref{fig:Edge_Setup}. Then we define the coloring up to $\op{PSL}_2$ by the condition:
\begin{equation} \label{eqn:cross}
\frac{(-1)^{r_v} \lambda_p\lambda_q\lambda_r^{-1}}{(-1)^{r_w} \lambda_p\lambda_t\lambda_s^{-1}} = (-1)^{r_v+r_w} \frac{\lambda_q\lambda_s}{\lambda_r\lambda_t} =: -\frac{(\kappa_A-\kappa_B)(\kappa_C-\kappa_D)}{(\kappa_B-\kappa_C)(\kappa_D-\kappa_A)}
\end{equation}
We shall often be agnostic towards the precise signs in this equation, instead writing $\pm$ to de-clutter the notation.

\begin{figure}[h]
\centering
\includegraphics[scale=0.55]{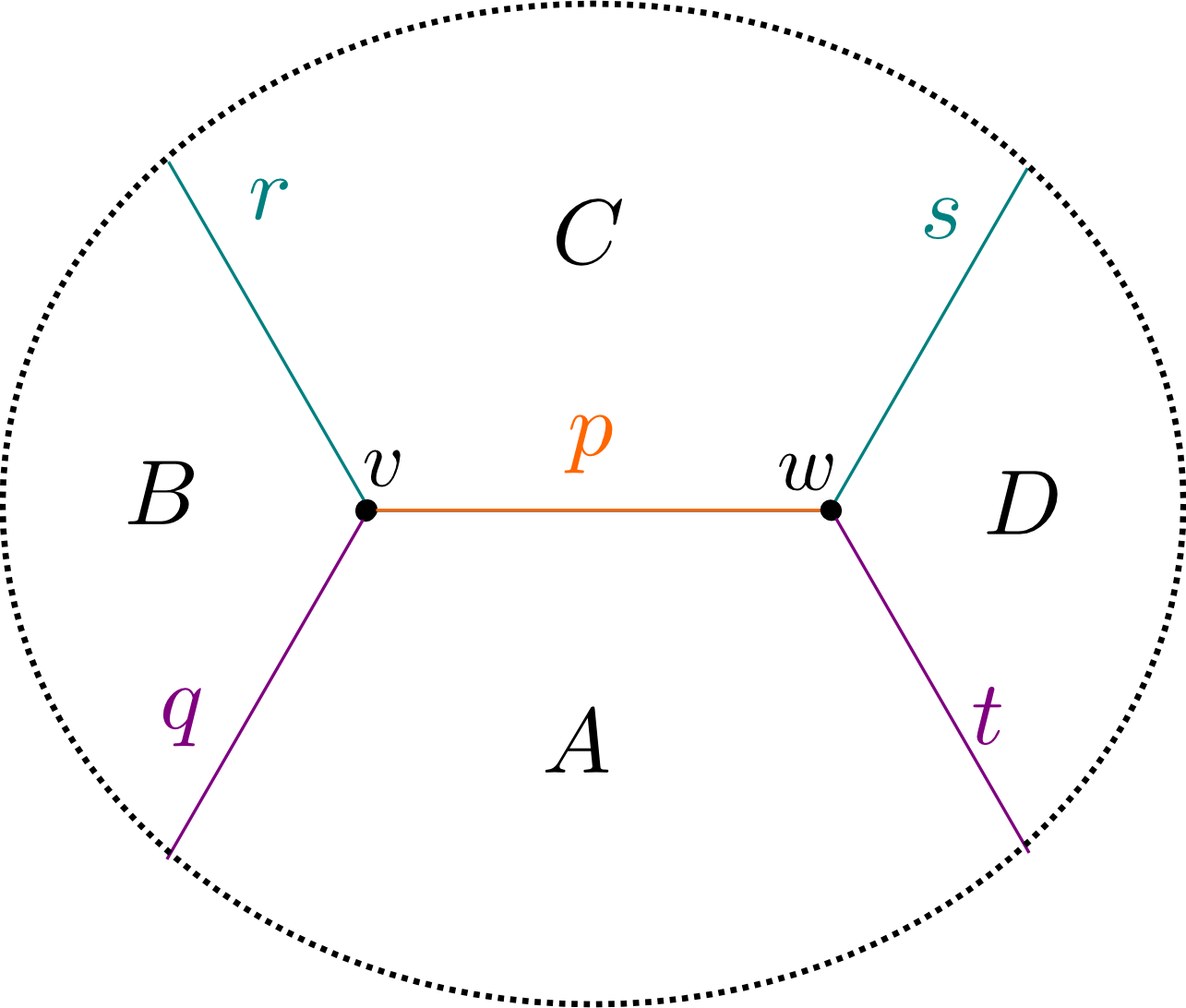}
\caption{The set-up for Equation \ref{eqn:cross}.}
\label{fig:Edge_Setup}
\end{figure}

If one of the $\kappa$ values is $\infty$, then Equation \ref{eqn:cross} is still valid in the following sense. Suppose that $\kappa_A = \infty$. Then we consider $(\kappa_A - \kappa_B)/(\kappa_D - \kappa_A) = -1$, leaving 
$$\pm \frac{\lambda_q\lambda_s}{\lambda_r\lambda_t} = \frac{\kappa_C-\kappa_D}{\kappa_B-\kappa_C}$$
If both $\kappa_B = \kappa_D = \infty$, then similarly, we take
$$\pm \frac{\lambda_q\lambda_s}{\lambda_r\lambda_t} = -1.$$
This interpretation of the right hand side yields the required $\op{PSL}_2$-independence - this quantity is often called the \textbf{cross ratio}. In particular, up to $\op{PSL}_2$ reparametrization, we can set $\kappa_A$, $\kappa_B$, and $\kappa_C$ to be any choice of distinct colors, and a choice of values for the $\lambda_i \in \F^*$ will uniquely determine $\kappa_D$ (different from $\kappa_A$ and $\kappa_C$ since the cross ratio is in $\F^*$). In addition, both sides of the equation are invariant under rotating the picture by 180 degrees, and so it is equivalent to determine $\kappa_D$ from $\kappa_A$, $\kappa_B$, and $\kappa_C$ as it is to determine $\kappa_B$ from $\kappa_A$, $\kappa_C$, and $\kappa_D$.

To be even more concrete, in practice, one can set the color of some given face to be $\infty$ up to $\op{PSL}_2$ reparametrization, and the ratios of the $\pm \lambda_m\lambda_n\lambda_k^{-1}$ values around the vertices are precisely the ratios of the differences between adjacent colors around the face. The fact that the sum of the $\pm \lambda_m\lambda_n\lambda_k^{-1}$ around a given face equals $0$ is precisely the fact that when we go around a face, we get back to the same color. An example of this picture around a triangular face is given in Figure \ref{fig:Local_Picture}.

\begin{figure}[h]
\centering
\includegraphics[scale=0.55]{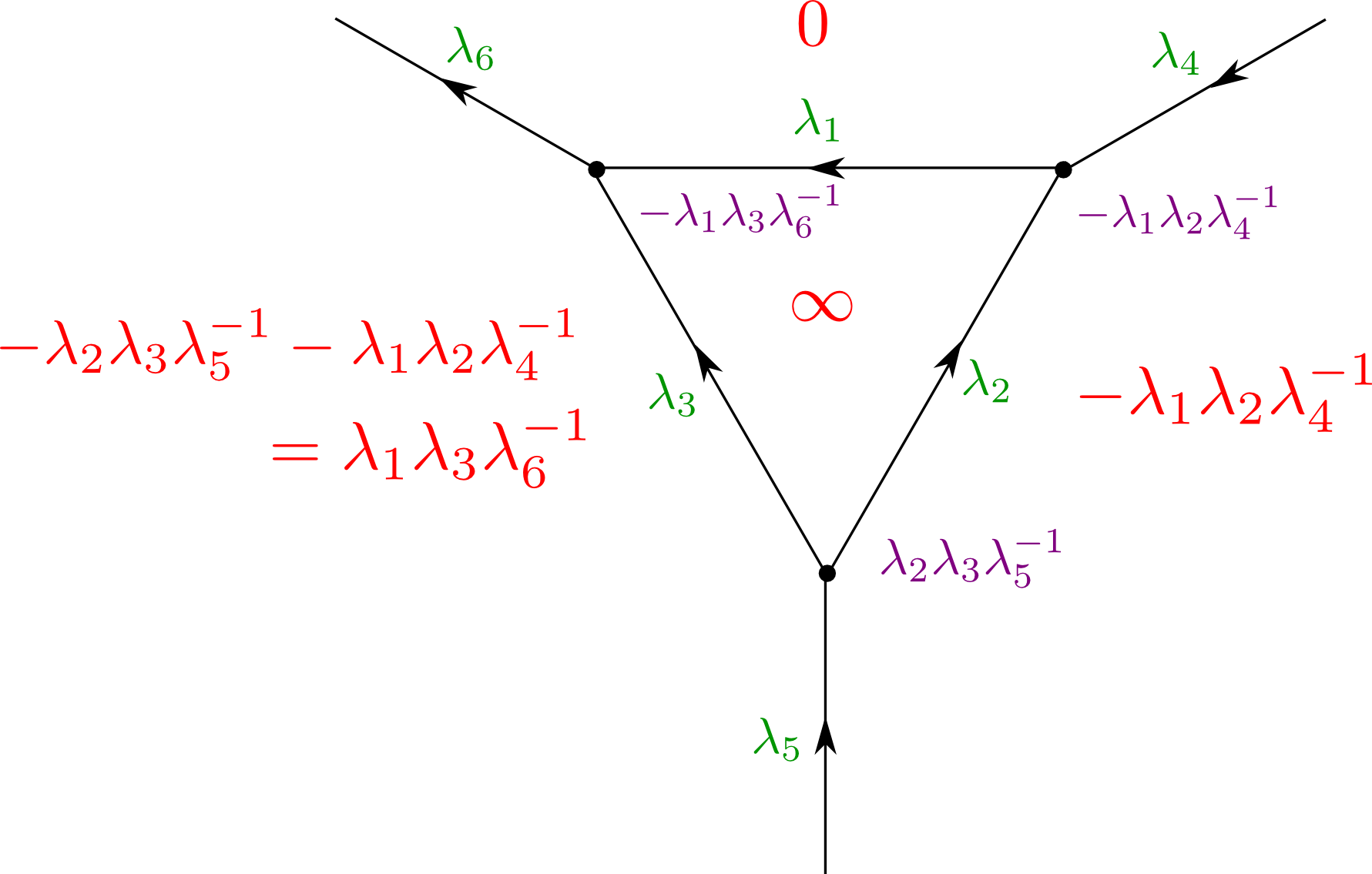}
\caption{An example of the coloring determined by a solution to the augmentation equations, locally around a triangular face. The arrows represent the chosen orientation of the graph. The green $\lambda_i$ values determine the coordinates in the augmentation variety. The purple values are the values of $\pm \lambda_m\lambda_n\lambda_k^{-1}$, with choice of sign determined by the orientation. And the red values up are, up to $\op{PSL}_2$ reparametrization, the colors of the faces determined by the augmentation.}
\label{fig:Local_Picture}
\end{figure}

If we know the distinct colors $\kappa_A,\kappa_B,\kappa_C$ around an edge, then we know the last color $\kappa_D$. So we can iteratively choose some edges where three of the colors around them are known, and fill in the last color. For example, suppose we take the dual graph $G^*$. The faces of $G$ correspond to the vertices of $G^*$, and so if in the dual graph we know the colors around the vertices of a triangle, then we can find the color of the last vertex of an adjacent triangle. This propagation procedure is depicted in Figure \ref{fig:Propagation}.

\begin{figure}[h]
\includegraphics[scale=0.85]{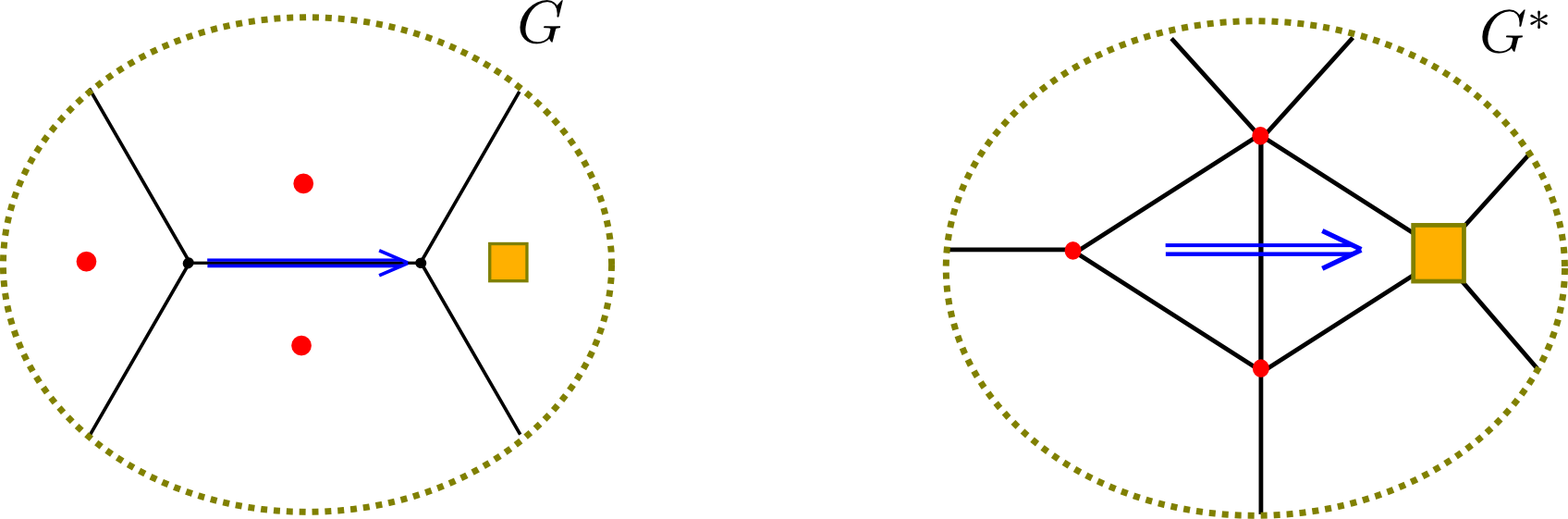}
\centering
\caption{Here we depict propagation through an edge, both in $G$ and $G^*$. The blue arrow is the direction of propagation, and is independent of the orientation choice along the edge. In both drawings, we are filling in the color in the orange square based on the colors at the red circles.}
\label{fig:Propagation}
\end{figure}

This gives a method from going from an augmentation in the augmentation variety of binary sequences to colorings. Start by labelling three adjacent faces around a fixed vertex of $G$ with the values $0,1,\infty$ (we could also use any other three distinct values in $\F \PP^1$, though we stick with these for concreteness). (We are assuming $G$ is bridgeless, so there is no issue with a face being adjacent to itself.) Then, using the propagation procedure described above, we can one-by-one fill in labels for all of the other faces, hence obtaining a labelling of the faces of $G$ with colors in $\F\PP^1$. It is not immediately obvious that this is a coloring, since it could be the case that when we propagated, we ended up with two adjacent faces having the same color. It is also not immediately clear that the labelling is well-defined up to $\op{PSL}_2(\F)$ reparametrization if we follow different directions of propagation to fill in our labels. That these two possibilities do not occur is implied by the following lemma.

\begin{lemma} \label{lem:aug_to_color} The labelling of faces with elements of $\F\PP^1$ obtained by this method satisfies Equation \ref{eqn:cross} around every edge. In particular, this defines a legitimate coloring, in that no two neighboring faces are labelled by the same color.
\end{lemma}

\begin{proof}
It suffices to prove that there exists a labelling of faces of $\F\PP^1$ satisfying Equation \ref{eqn:cross} around every edge. Then we note that since any other choice of propagation requires Equation \ref{eqn:cross} to hold along some subset of edges but that this coloring is uniquely determined by these conditions, the resulting labelling must be the same.

The way we go about this is to abstractify the way in which a face is labelled via some connected sequence of propagations. Let us formalize this in the following way. Suppose that we choose any vertex $v$ and any face $f$. This corresponds in the dual graph $G^*$, which we embed into $S^2$, to a dual face $\wt v$ and a dual vertex $\wt f$. Further suppose that the vertices around $\wt v$ are labelled by $0,1,\infty$. Consider any path $\gamma \colon [0,1] \rightarrow S^2$ such that the following three conditions are satisfied:
\begin{itemize}
	\item[-] $\gamma(0)$ lies at some fixed interior point in $\wt v$
	\item[-] $\gamma(1)$ is in $\wt f$
	\item[-] $\gamma$ never passes through the vertex of $G^*$ corresponding to the face at infinity
\end{itemize}
Given such a path, we shall obtain a label for $\wt v$ by an element of $\F\PP^1$ essentially by keeping track of local labels for the vertices of $G^*$ which are near $\gamma(t)$. To state this precisely, we show that $\gamma$ defines a labelling function $L_{\gamma} \colon V(G^*) \times [0,1] \rightarrow \F\PP^1 \cup \{*\}$, where $V(G^*)$ is the vertex set of $G^*$. This $L_{\gamma}$ is precisely determined by the following conditions:
\begin{itemize}
	\item[(i)] We set the initial values $L_{\gamma}(w, 0) \in \{0,1,\infty,*\}$ is defined by the initial labelling of vertices around $\wt v$. That is, $L_{\gamma}(w,0)$ is $0,1,\infty$ for the three initially labelled vertices, and is $*$ on all other vertices.
	\item[(ii)] We set $L_{\gamma}$ equal to $*$ unless $w$ is near $\gamma(t)$ in the following sense:
	\begin{itemize}
		\item If $\gamma(t)$ is in the interior of a face of $G^*$ then there are three labels which are not $*$, which are the vertices around this face.
		\item If $\gamma(t)$ is in the interior of an edge of $G^*$ then there are four labels which are not $*$, which are the four vertices around this edge. In this case, we require the labels to satisfy Equation \ref{eqn:cross} with respect to the augmentation at this edge.
		\item If $\gamma(t)$ is at a vertex of $G^*$, then the only labels which are not $*$ are the vertex itself and all adjacent vertices. Furthermore, we require all of these labels to satisfy Equation \ref{eqn:cross} at all edges connecting the vertex $\gamma(t)$ to an adjacent vertex. The augmentation equations imply that such a labelling is possible.
	\end{itemize}
	\item[(iii)] Each $L_{\gamma}(w,\cdot) \colon [0,1] \rightarrow \F\PP^1 \cup \{*\}$ is locally constant on its restriction to the preimage of $\F\PP^1$.
\end{itemize}

To summarize, the first condition states that $L_{\gamma}$ has an initial value, while the second and third together imply that $L_{\gamma}$ only changes when $\gamma$ switches strata, and the corresponding labels we keep track of are completely determined by Equation \ref{eqn:cross}. In this way, $L_{\gamma}$ can be constructed and is uniquely defined by these conditions. We provide an example in Figure \ref{fig:Propagation_Path}.

\begin{figure}[h]
\includegraphics[scale=1.3]{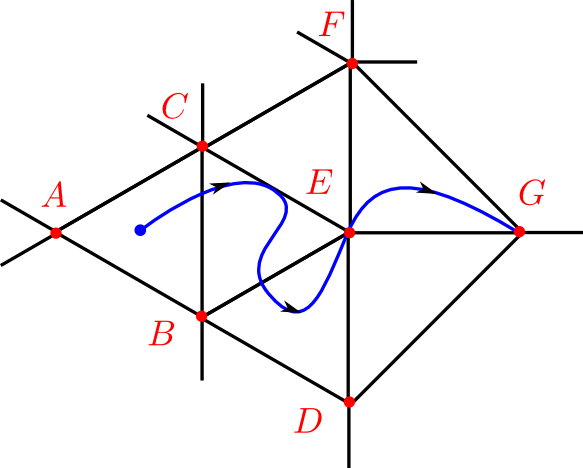}
\centering
\caption{In this figure, we consider the blue path $\gamma$, and suppose that $A$, $B$, and $C$ are labelled $0$, $1$, and $\infty$ respectively. Then when $\gamma$ hits segment $BC$, $L_{\gamma}$ also includes the label for $E$ such that Equation \ref{eqn:cross} is satisfied. When $\gamma$ then enters the interior of triangle $BCE$, it forgets the label of $A$. Continuing on, it momentarily includes a label for $F$ which it then forgets, then adds on a label for $D$, and then forgets the label for $C$. When $\gamma$ passes through $E$, it then extrapolates the labels for $B$, $D$, and $E$ into labels for also $C$, $F$, and $G$. Finally, it forgets everything except for $E$, $F$, and $G$, and then gains an extra label (out of the picture) at its terminal point. The associated value $L_{\gamma}(G,1)$ is the label we are interested in.}
\label{fig:Propagation_Path}
\end{figure}

So to each path $\gamma$ as above, we can associate to $\wt f$ the value $L_{\gamma}(\wt f,1) \in \F\PP^1$. We wish to show that this final label is invariant of $\gamma$. In order to do this, note that all isotopies of $\gamma$ are generated by small isotopies of $\gamma$ on some interval $[a,b]$ which completely lie in the star of some $w \in V(G^*)$ with $w$ not corresponding to the face at infinity of $G$. Under such an isotopy, $L_{\gamma}(\cdot,b)$ is preserved since one could keep track also of the labels for all the vertices around $w$ which would satisfy Equation \ref{eqn:cross} around all edges adjacent to $w$. Hence also $L_{\gamma}(\cdot,\geq b)$ is preserved since the construction of $L_{\gamma}$ forgets all labels which occur before time $t = b$. In particular, $L_{\gamma}(\wt f, 1)$ is always preserved by isotopy. Finally, the space of paths $\gamma$ considered is contractible, so this yields a well-defined label for $\wt f$.

One could then repeat these for all other faces $f$ to obtain a complete labelling depending on the value of $\gamma(1)$ being considered. So consider we take $\gamma$ ending at $\wt f$ and modify it to $\gamma'$ by adding an arc along the edge from $\wt f$ to an adjacent $\wt f'$. Then $L_{\gamma}(\wt f',1) = L_{\gamma'}(\wt f',1)$, so $\gamma$ itself encodes the same label for $\wt f'$. But we see that $L_{\gamma}(\cdot,1)$ consists of labels which satisfy Equation \ref{eqn:cross} at every edge around $\wt f$, and so the labelling we have obtained satisfies this equation across all edges.
\end{proof}

\begin{remark}
Since Equation \ref{eqn:cross} is satisfied around all of the edges of the face at infinity $f_{\infty}$, in fact this also proves that the equation $\sum_{v \in f_{\infty}} H(\tau_{f_{\infty}}(v)) = 0$ is also satisfied by elements in the augmentation variety.
\end{remark}

\begin{cor}
This yields a map of varieties $\Phi \colon X_T(\F) \rightarrow \chi_{G^*}(\F\PP^1)$, which is independent of the choices made in the propagation procedure (the starting vertex and the directions of propagation).
\end{cor}

\begin{proof} Each time we propagate in our procedure, $\kappa_D$ is a rational function of $\kappa_A$, $\kappa_B$, and $\kappa_C$, so we do obtain a morphism of varieties $\wt{\Phi} \colon \widetilde{X_G(\F)} \rightarrow \chi_{G^*}(\F\PP^1)$. Precomposing with $X_T(\F) \hookrightarrow \wt{X_G(\F)}$ yields our morphism $\Phi \colon X_T(\F) \rightarrow \chi_{G^*}(\F)$. Independence of propagation is immediate from satisfying Equation \ref{eqn:cross} around every edge.
\end{proof}

\begin{lemma} The morphism $\Phi \colon X_T(\F) \rightarrow \chi_{G^*}(\F\PP^1)$ is injective as a map of sets.
\end{lemma}
\begin{proof}
We prove this in two steps. Suppose we have a coloring in the image of $\Phi$.
\begin{enumerate}
	\item We will first prove that the information of each $\pm \lambda_m\lambda_n\lambda_k^{-1}$ (each $H(\tau_f(v))$) is uniquely determined.
	\item Second, we use this data to uniquely determine all of the values of the $\lambda_i$.
\end{enumerate}

For the first step, note that if we know the value of $\lambda_m\lambda_n\lambda_k^{-1}$ for one face-vertex adjacency, then given the coloring, we know these values around the rest of that face by applying Equation \ref{eqn:cross} around the edges of the face. Also note that we know this value to be $\pm 1$ around the chosen vertex $v_0$ in the basis of the graph, and hence around the three faces around this vertex. We claim that we can always find the $\lambda_m\lambda_n\lambda_k^{-1}$ values inductively by choosing an adjacent face to the region where we have already found these values. Let $R$ be this region. Then note that since the tree $T$ contains no cycle, there is some edge $e$ of $\partial R$ which is not in the tree (so that $\lambda = 1$), and which is adjacent to the faces $f \in R$ and $g \notin R$. Pick one of the vertices $v$ bounding $e$, and $e_1$ and $e_2$ the other edges around $v$. Then $H(\tau_f(v)) = \pm \lambda_1/\lambda_2$ and $H(\tau_g(v))=\pm \lambda_2/\lambda_1$, with the same sign. These are therefore just reciprocals of each other. But we already know $H(\tau_f(v))$ since $f$ lies in $R$, so this allows us to extend this information through to the face $g$, which was not in $R$.

For the second step, note first that $\lambda_i = 1$ for non-tree edges. Meanwhile, if we know two values $\lambda_i$ and $\lambda_j$ around a vertex, then $\lambda_k$ is uniquely determined. So we can start at the roots of the tree $T$ and uniquely determine the value along the corresponding edge, and continue to deplete the tree in this way until all edges have been filled in uniquely. This proves the result.
\end{proof}

Finally, identifying $X_G(\F) \simeq X_T(\F)$ yields the desired morphisms $\Phi \colon X_G(\F) \rightarrow \chi_{G^*}(\F\PP^1)$. However, it is not obvious that this map is canonical, since a priori, it could depend on the choice of tree $T$. In order to prove it is canonical, we need to understand how these morphisms depend on $T$. In the introduction to this appendix, we constructed the projection morphism $\pi \colon \wt{X_G(\F)} \rightarrow X_T(\F)$. Our key lemma is the following.

\begin{lemma} \label{lem:nat mor} The morphism $\Phi \colon X_T(\F) \rightarrow \chi_{G^*}(\F\PP^1)$ satisfies $\Phi \circ \pi = \wt{\Phi}$.
\end{lemma}
\begin{proof}
Recall that $\pi(\vec{\lambda}) = \vec{\lambda}'$ with $\lambda'_k = \lambda_k \prod_{j=1}^{g+3}\lambda_j^{b_{kj}}$. We need to check that applying $\Phi$ to this vector yields the same value as applying $\wt{\Phi}$ to the original vector. But the construction of $\wt{\Phi}$ only depended on the values appearing on the left hand side of Equation \ref{eqn:cross}. Hence, it suffices to prove that around every edge $p$, using the same notation as the set-up from before (see Figure \ref{fig:Local_Picture}), we have
$$\frac{\lambda'_q\lambda'_s}{\lambda'_r\lambda'_t} = \frac{\lambda_q\lambda_s}{\lambda_r\lambda_t}.$$
Writing this out in terms of our expression for $\lambda'$, we find that this is equivalent to checking that
$$\prod_{j=1}^{g+3}\lambda_j^{b_{qj}+b_{sj}-b_{rj}-b_{tj}} = 1,$$
or in other words, that for each $j = 1,\ldots,g+3$, the matrix $b_{ij}$ satisfies $b_{qj}+b_{sj}-b_{rj}-b_{tj} = 0$ when $p,q,r,s$ are in the position corresponding to our set-up.

Recall that
$$
\left(\rule{0cm}{0.5cm}\quad b_{ij} \quad \right) = - \left(\begin{array}{cc}
A\\
\,
A_T
\end{array}\right) \left(\rule{0cm}{0.5cm} \quad A^{-1} \quad \right).
$$
But the matrix $\left(\begin{array}{cc}A\\ \, A_T\end{array}\right)$ is just the face-edge adjacency matrix beginning with a column of all $1$'s. Hence, if we multiply on the left with the row with entries $+1$ corresponding to edges $q$ and $s$ and $-1$ corresponding to edges $r$ and $p$, we see that the result is the $0$ row vector, since each of the regions $A,B,C,D$ in Figure \ref{fig:Local_Picture} borders one edge contributing $+1$ and one contributing $-1$. Therefore, this row times the matrix $b_{ij}$ yields $0$, but the entries of this row are precisely the entries of the form $b_{qj}+b_{sj}-b_{rj}-b_{tj}$.
\end{proof}

\begin{cor} The morphism $\Phi \colon X_G(\F) \rightarrow \chi_{G^*}(\F\PP^1)$ does not depend on the choice of tree $T$.
\end{cor}
\begin{proof}
For any two choices of tree $T$ and $T'$, we have the following commutative diagram:
$$\begin{tikzcd}
X_T \arrow[equal]{rr} \arrow[hookrightarrow]{rd}{i} & & X_T \ar{rd}{\Phi} & \\
& \wt{X_G} \ar{rr}{\wt{\phi}} \ar{ru}{\pi} \ar{rd}{\pi} &  & \chi_{G^*} \\
X_{T'} \arrow[equal]{rr} \arrow[hookrightarrow]{ur}{i} & & X_{T'} \ar{ru}{\Phi'}] &
\end{tikzcd}$$
But then, $\Phi = \Phi \circ \pi \circ i = \wt{\Phi} \circ i = \Phi' \circ \pi' \circ i$, and finally we obtain the independence since $\pi' \circ i \colon X_T \rightarrow X_{T'}$ is the morphism $X_T \xrightarrow{\sim} X_G \xrightarrow{\sim} X_{T'}$ by the following commutative diagram:
$$\begin{tikzcd}
X_T \ar[hookrightarrow]{d}{i} \ar[equal]{r} & X_T \\
\wt{X_G} \ar{ur}{\pi} \ar{r} \ar[swap]{rd}{\pi'}] & X_G \ar[swap]{u}{\sim}] \ar{d}{\sim} \\
 & X_{T'}
\end{tikzcd}$$
\end{proof}

\subsection{Construction of an inverse morphism}

\begin{lemma} For $\F = \overline{\F}$ algebraically closed, then $\Phi$ is an isomorphism.
\end{lemma}
\begin{proof}
We shall again take the viewpoint that $X_G(\F) \simeq X_T(\F)$, and construct an inverse morphism $\Psi \colon \chi_{G^*}(\F\PP^1) \rightarrow X_T(\F)$. In order to do this, we first define $\Psi$ on affine pieces of $\widetilde{\chi_{G^*}(\F\PP^1)}$. The image of each of these pieces will lie in $\widetilde{X_G(\F)}$, but will not patch together properly. However, the construction is done equivarently so that they do indeed patch together to descend to a map $\chi_{G^*}(\F\PP^1) \rightarrow X_T(\F)$. In other words, the images do patch in $X_T(\F)$, and the image only depends on the coloring up to $\op{PSL}_2(\F)$ reparametrization.

For any element $x \in \F\PP^1$, pick an element $\xi_x \in \op{PSL}_2(\F)$ sending $x$ to $\infty$. We define the open affine piece $U_x$ of the pre-chomatic variety as
$$U_x = \widetilde{\chi_{G^*}(\F\PP^1)} \cap (\F\PP^1 \setminus \{x\})^{g+3}.$$
The element $\xi_x$ then naturally identifies $U_x$ as an affine variety $V_x \subseteq \F^{3g+3}$, where the correspondence $\xi_x$ preserves the image in $\chi_{G^*}(\F\PP^1)$. We write
$$\xi_x \colon U_x \xrightarrow{\sim} V_x.$$
So long as $|\F| \geq g+3$, the open sets $U_x$ cover $\wt{\chi_{G^*}(\F\PP^1)}$ (since every coloring uses at most $g+3$ elements in $\F\PP^1$). In our case, $\F$ is algebraically closed, hence infinite.

For each $x$, we form the morphism $\wt{\Psi}_x \colon V_x \rightarrow (\F^*)^{3g+3}$ as follows. Consider some coloring $\kappa \in V_x$. Then take $\lambda_k := \kappa_A - \kappa_B$ corresponding to the orientation of the edge $k$, where $A$ and $B$ are the regions above and below $k$, respectively, when $k$ is oriented to the right. We see then immediately that with these choices of $\lambda_k$, Equation \ref{eqn:cross} is satisfied at every edge. In particular, the image of $\wt{\Psi}_x$ lies in $\wt{X_G(\F)}$ and satisfies
$$\wt{\Phi}\wt{\Psi}_x(\kappa) = [\kappa] \in \chi_{G^*}(\F\PP^1)$$
by construction. Now set $\pi \colon \wt{X_G(\F)} \rightarrow X_T(\F)$ the natural projection. Then by Lemma \ref{lem:nat mor}, the previous equation implies
$$\Phi \circ \pi \circ \wt{\Psi}_x(\kappa) = [\kappa].$$
In particular, if $\kappa \in U_x$, then $[\xi_x(\kappa)] = [\kappa]$, and so
$$\Phi \circ \pi \circ \wt{\Psi}_x \circ \xi_x(\kappa) = [\kappa].$$
But $\Phi$ is injective, so the maps $\pi \circ \wt{\Psi}_x \circ \xi_x$ defined one each $U_x$ glue together to give some morphism $\wt{\Psi} \colon \wt{\chi_{G^*}(\F\PP^1)} \rightarrow X_T(\F)$ satisfying
$$\Phi \wt{\Psi}(\kappa) = [\kappa].$$
In particular, again since $\Phi$ is injective, $\wt{\Psi}$ only depends on the class of $\kappa$, and so descends to $\Psi \colon \chi_{G^*}(\F) \rightarrow X_T(\F)$ such that $\Phi \Psi$ is equal to the identity on $\chi_{G^*}(\F)$. Therefore, $\Phi$ must also be surjective on rational points, and so it is a bijection, and so $\Psi$ is in fact its inverse map as a map of sets. But both of these are rational maps, and since we are working over an algebraically closed field, the variety is completely described by its set of points, and so this is an isomorphism of varieties.
\end{proof}

The above proof goes through to show that $\Phi$ induces a bijection of rational points so long as the open sets $U_x$ cover $\wt{X_G(\F)}$. So for $|\F| \geq g + 3$, if we were only interested in establishing this bijection, we would be finished. Establishing an isomorphism of varieties in all cases, including when $|\F| < g+3$, is the following result.

\begin{cor} \label{cor:conclusion} For any $\F$, $\Phi$ is an isomorphism of $\F$-varieties.
\end{cor}
\begin{proof}
This is a consequence of \cite[Proposition 2.7.1(viii)]{GD}, noting that the the natural extension of $\Phi_{\F}$ to $\overline{\F}$ varieties (which is just the same $\Phi_{\overline{\F}}$ we defined) is an isomorphism of $\overline{\F}$-varieties.
\end{proof}

This completes the proof of Theorem \ref{thm:augs_colors}.

\subsection{Examples} \label{subsect:examples}
We present two examples to illustrate the correspondence between augmentations and colorings.

In the first example, depicted in Figure \ref{fig:Color_example_mod_3}, we work over the field $\F = \F_3$. Hence, $\F^*$ consists of two elements, $\pm 1$, and $\F\PP^1$ consists of the elements $\{0,1,2,\infty\}$. Note that the information of $H(\tau_f(v)) = \pm \lambda_m \lambda_n \lambda_k^{-1}$ actually only depends on $v$, since $\F^*$ is invariant under inversion, so $\lambda_k^{-1} = \lambda_k$. Hence, it suffices to only label $H(\tau_f(v))$ once around each vertex $v$ as opposed to three times. A sample augmentation has values in green, with the information of $H(\tau_f(v))$ in purple corresponding to the orientation in light blue. A corresponding (pre-)coloring is written in red. Note that there is a vertex around which the faces are $0,1,\infty$, and the reader can find an edge $T$ such that this example corresponds to a $T$-versal augmentation.

\begin{figure}[h]
\includegraphics[scale=0.7]{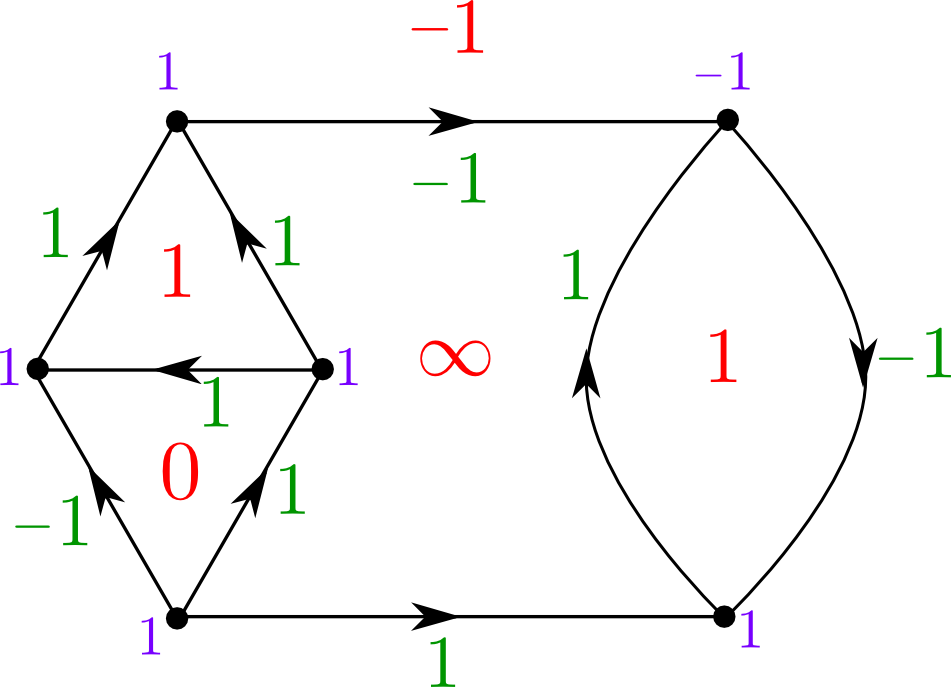}
\centering
\caption{In this example, over $\F_3$, the arrows represent the orientation, the green labels at the edges represent some solution to the augmentation equations, the purple labels at the vertices are the values of $H(\tau_f(v)) = \pm \lambda_m\lambda_n\lambda_k^{-1}$, and the red labels are a corresponding coloring (determined up to $\op{PSL}_2(\F_3)$)}
\label{fig:Color_example_mod_3}
\end{figure}

In the second example, depicted in Figure \ref{fig:Color_example_mod_2}, we work over the field $\F = \F_4$, with elements $\{0,1,\alpha,\alpha+1\}$, satisfying $\alpha^2 + \alpha + 1 = 0$. Since this field has characteristic $2$, signs become unimportant, so we do not need to include a choice of orientation. However, the labels of $H(\tau_f(v))$ do depend on the choice of face around the vertex $v$, so we need to label all three corners around a vertex. We use the same colors as in the previous example.

\begin{figure}[h]
\includegraphics[width=8cm]{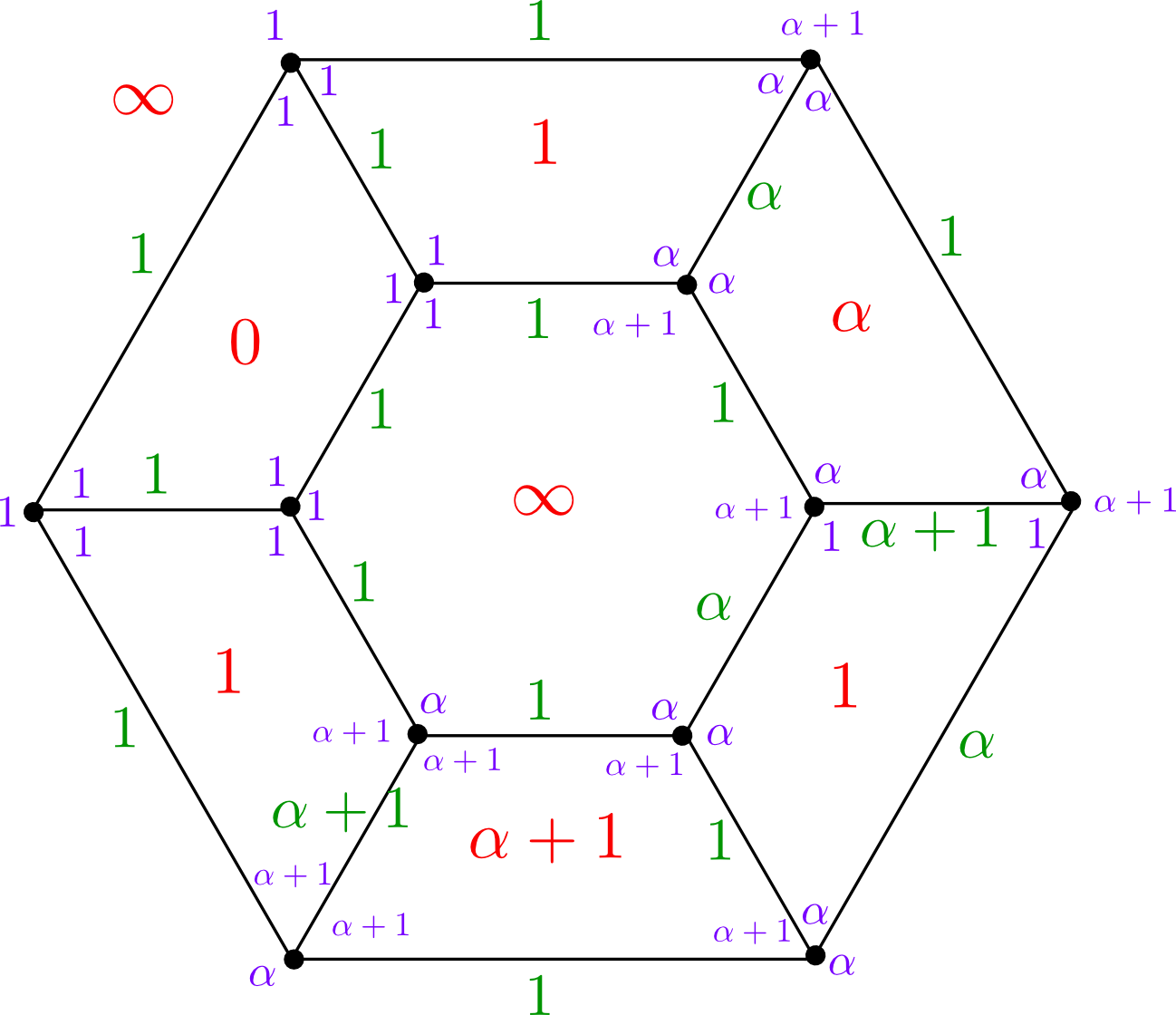}
\centering
\caption{Another example, this one over $\F_4 = \F_2[\alpha]/(\alpha^2 + \alpha + 1)$, where no orientation is needed, but the vertices are labelled thrice corresponding to the different values of $H(\tau_f(v))$ around a vertex $v$.}
\label{fig:Color_example_mod_2}
\end{figure}

\subsection{A related combinatorial identity} \label{subsect:numtheory}
A key idea in the proof of the theorem was that the precise values of the $\pm \lambda_m\lambda_n\lambda_k^{-1}$ were not relevant, but rather the ratios between them around a given face. Consider to each face-vertex adjacency $(f,v)$, we assign some value $\mu_{f,v} \in \F^*$, of which there are $3|V| = 6g+6$ values to assign. One thing we could to is try to impose conditions on the $\mu_{f,v}$ so that we precisely recreate what the coloring conditions coming from the cross ratio being the ratio of consecutive $\mu$ values. So suppose we have an edge $e$ between vertices $v$ and $w$ separating faces $f$ and $g$. Then recall that the cross ratio should be encoded by both $\mu_{f,v}/\mu_{f,w}$ and $\mu_{g,w}/\mu_{g,v}$. So the first condition is that for each edge $e$ as above,
\begin{equation} \label{eqn:mu_prod}
\mu_{f,v}\mu_{g,v} = \mu_{f,w}\mu_{g,w}
\end{equation}

Additionally, one should have (as from the augmentation equations) that
\begin{equation} \label{eqn:mu_sum}
\sum_{v \in f} \mu_{f,v} = 0
\end{equation}
for each face $f$.

\begin{definition}
We define the variety $\widetilde{M_G(\F)} \subset (\F^*)^{6g+6}$ by the variety defined by Equations \ref{eqn:mu_prod} and \ref{eqn:mu_sum}. These equations are invariant by multiplying all $\mu$ around a given face by an element of $\F^*$, or by multiplying all of them by an element of $\F^*$, and the quotient $M_G(\F) = \widetilde{M_G(\F)}/(\F^*)^{g+3}$ is called the \textbf{$\mu$-color variety}.
\end{definition}

In essence, Lemma \ref{lem:aug_to_color} can be thought of as producing a morphism $M_G(\F) \rightarrow \chi_{G^*}(\F\PP^1)$, through which our morphism $X_G(\F) \rightarrow \chi_{G^*}(\F\PP^1)$ naturally factors. Implicit in this proof is the following proposition.

\begin{prop}
As varieties, the natural morphism $M_G(\F) \rightarrow \chi_{G^*}(\F\PP^1)$ is an isomorphism.
\end{prop}

\begin{cor}
The natural morphism $X_G(\F) \rightarrow M_G(\F)$ is also an isomorphism.
\end{cor}

For an element $\mu \in \widetilde{M_G(\F)}$ and an edge $e$, we shall set
$$\nu_e := \mu_{f,v}\mu_{g,v} = \mu_{f,w}\mu_{g,w}$$
using the same notation of Equation \ref{eqn:mu_prod}.

\begin{remark}
There is also a natural morphism $\widetilde{X_G(\F)} \rightarrow \wt{M_G(\F)}$, but it is not generally an isomorphism since it is not necessarily surjective on rational points. Indeed, surjectivity would require $\lambda_e^2 = \nu_e$. So if $\F$ is not closed under square roots, then there is no hope for this to be an isomorphism.
\end{remark}

We present the following neat combinatorial identity coming from comparing this viewpoint. However, despite the simple nature of the identity, the author is not aware of a purely combinatorial way of obtaining the correct sign in the statement without first lifting to an augmentation.
\begin{cor} \label{cor:mu_col}
Given an element $\mu \in \widetilde{M_G(\F)}$, one has
$$\prod_{e} \nu_e = (-1)^{g+1}\prod_{(f,v)}\mu_{f,v}$$
where the left-hand product is taken over all edges and the right-hand product is taken over all face-vertex adjacencies.
\end{cor}
\begin{proof}
We have that $\mu$ descends to some element of $M_G(\F)$, and hence comes from an augmentation $\lambda \in X_G(\F)$, which in turn has some lift to some $\wt\lambda \in \wt{X_G(\F)}$. This maps to some $\mu' \in \wt{M_G(\F)}$ representing the same coloring as $\mu$. Meanwhile, the equation we wish to prove is invariant under the action of $(\F^*)^{g+3}$, and so it suffices to prove the equation holds for $\mu'$ and the corresponding $\nu'$. But in this case, we can write $\mu'$ and $\nu'$ in terms of the $\wt\lambda$ values. This yields
\begin{eqnarray*}
\prod_{(f,v)}\mu'_{f,v} &=& \prod_{(f,v)} (-1)^{r_v}\wt\lambda_m\wt\lambda_n\wt\lambda_k^{-1} \\
	&=& \prod_v (-1)^{3r_v} (\wt\lambda_m\wt\lambda_n\wt\lambda_k^{-1})(\wt\lambda_n\wt\lambda_k\wt\lambda_m^{-1})(\wt\lambda_k\wt\lambda_m\wt\lambda_n^{-1}) \\
	&=& \prod_v (-1)^{r_v} \wt\lambda_m\wt\lambda_n\wt\lambda_k \\
	&=& \prod_v (-1)^{r_v} \prod_e \wt\lambda_e^2 \\
	&=& \prod_v (-1)^{r_v} \prod_e \nu'_e
\end{eqnarray*}
where we write $e_m$, $e_n$, and $e_k$ for the edges around $v$. Now recall $r_v$ is just the number of outgoing edges at $v$. So one can think of $(-1)^{r_v} = (-1)^{r_m}(-1)^{r_n}(-1)^{r_k}$ where $r_i$ is $1$ if $e_i$ is outgoing and $0$ otherwise. Hence, since each edge is outgoing in one vertex and ingoing in another, each edge contributes a minus sign to the total product. There are $3g+3$ edges, so $\prod_v (-1)^{r_v} = (-1)^{3g+3} = (-1)^{g+1}$, and this completes the claim.
\end{proof}
\section{Vertex crossing with signs}\label{app:signs}
In this appendix we verify that a tine-vertex crossing move leaves the differential invariant. We do this for all possible configurations of orientation.

\begin{figure}[h!]
\centering
  \includegraphics[scale=0.65]{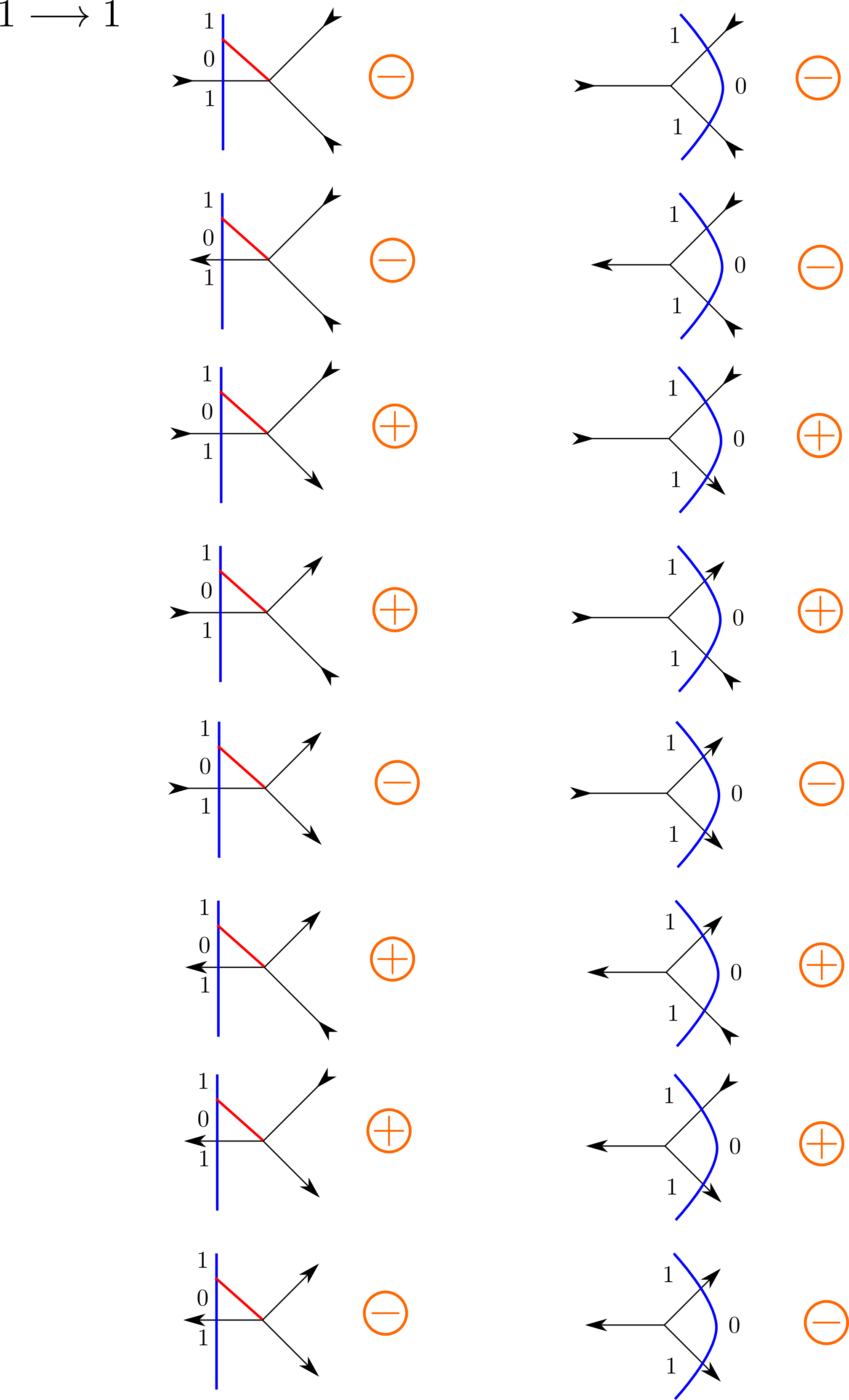}
  \caption{Signs I}
\end{figure}
\begin{figure}[h!]
\centering
  \includegraphics[scale=0.65]{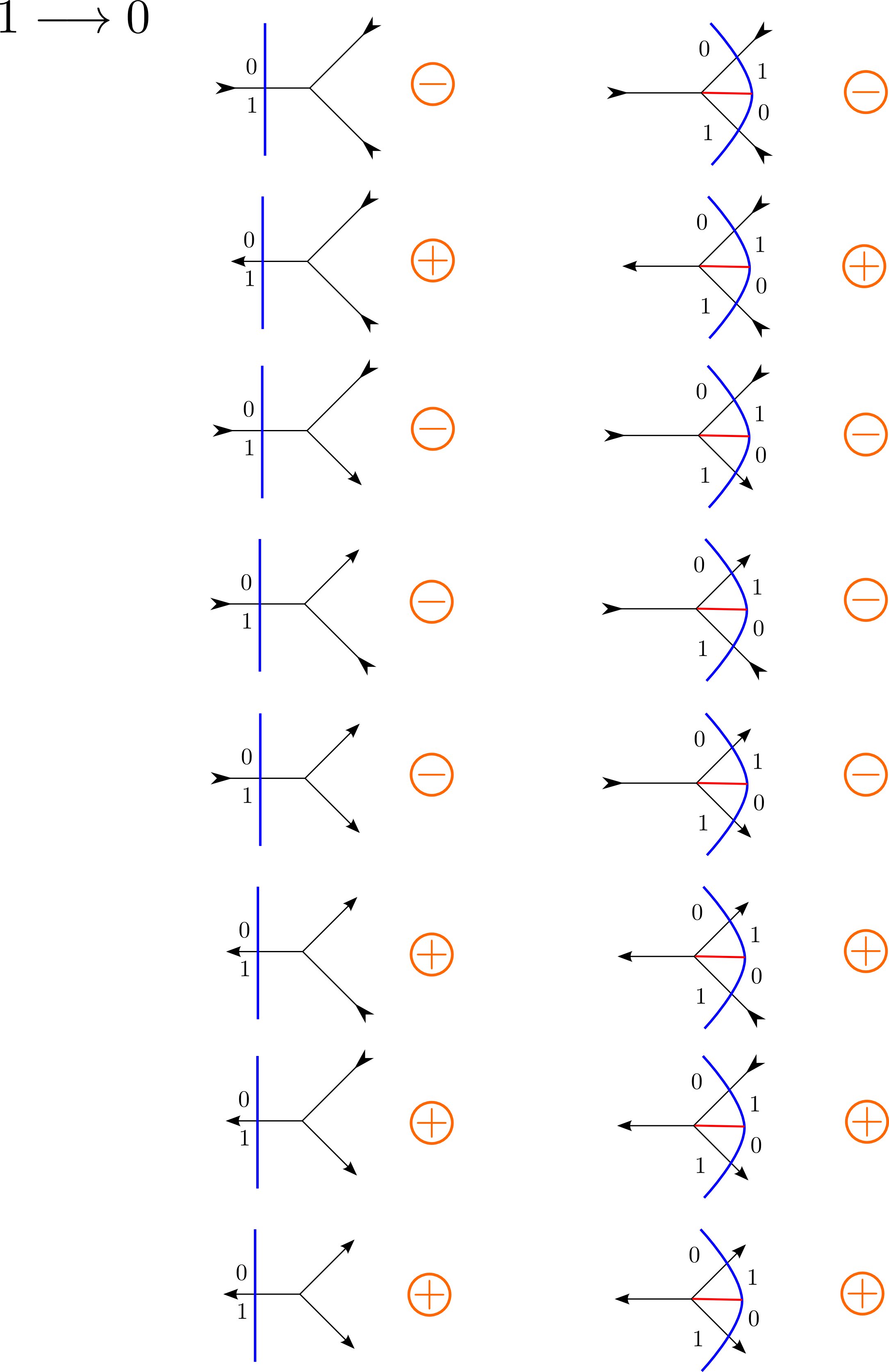}
  \caption{Signs II}
\end{figure}
\begin{figure}[h!]
\centering
  \includegraphics[scale=0.65]{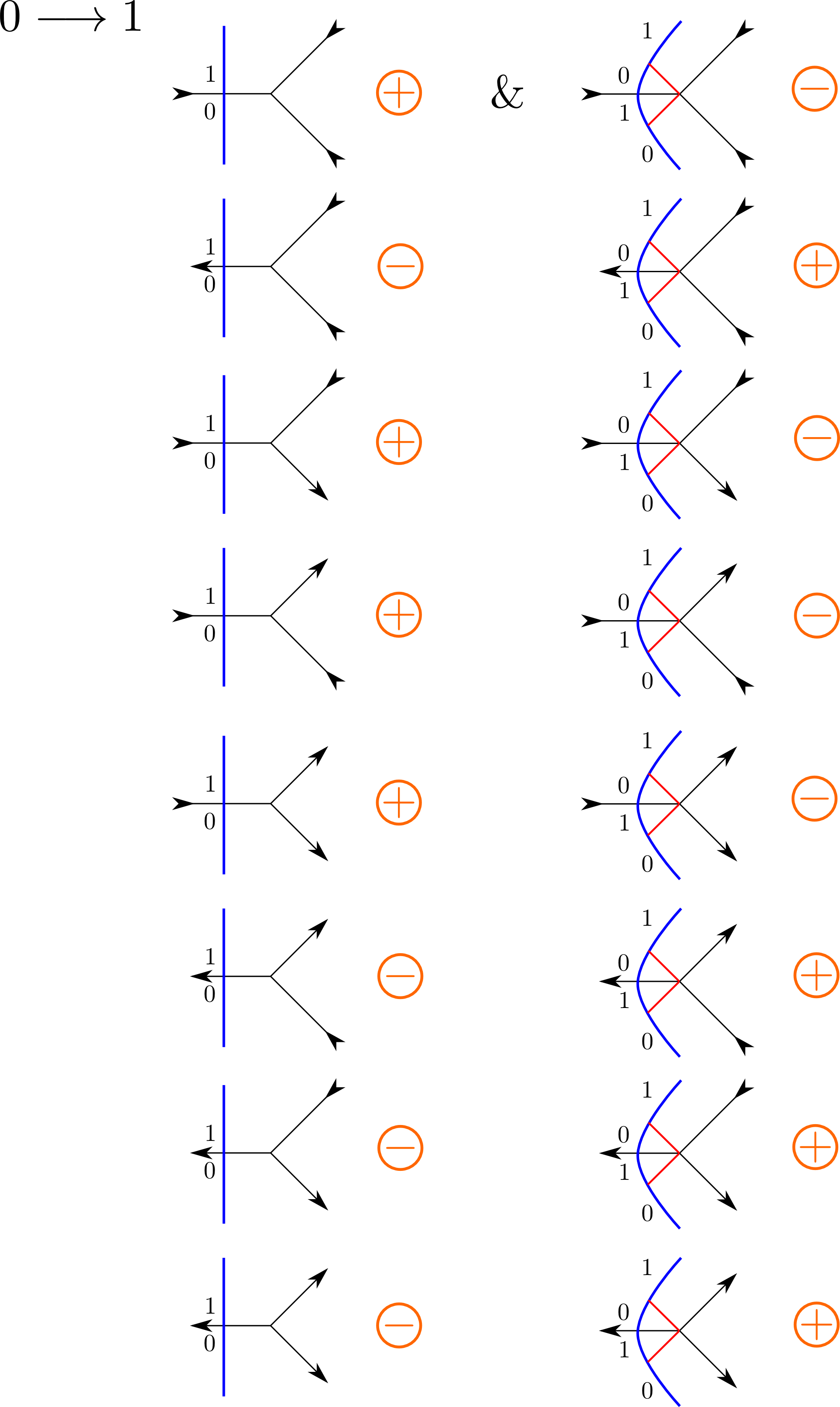}
  \caption{Signs III}
\end{figure}
\begin{figure}[h!]
\centering
  \includegraphics[scale=0.65]{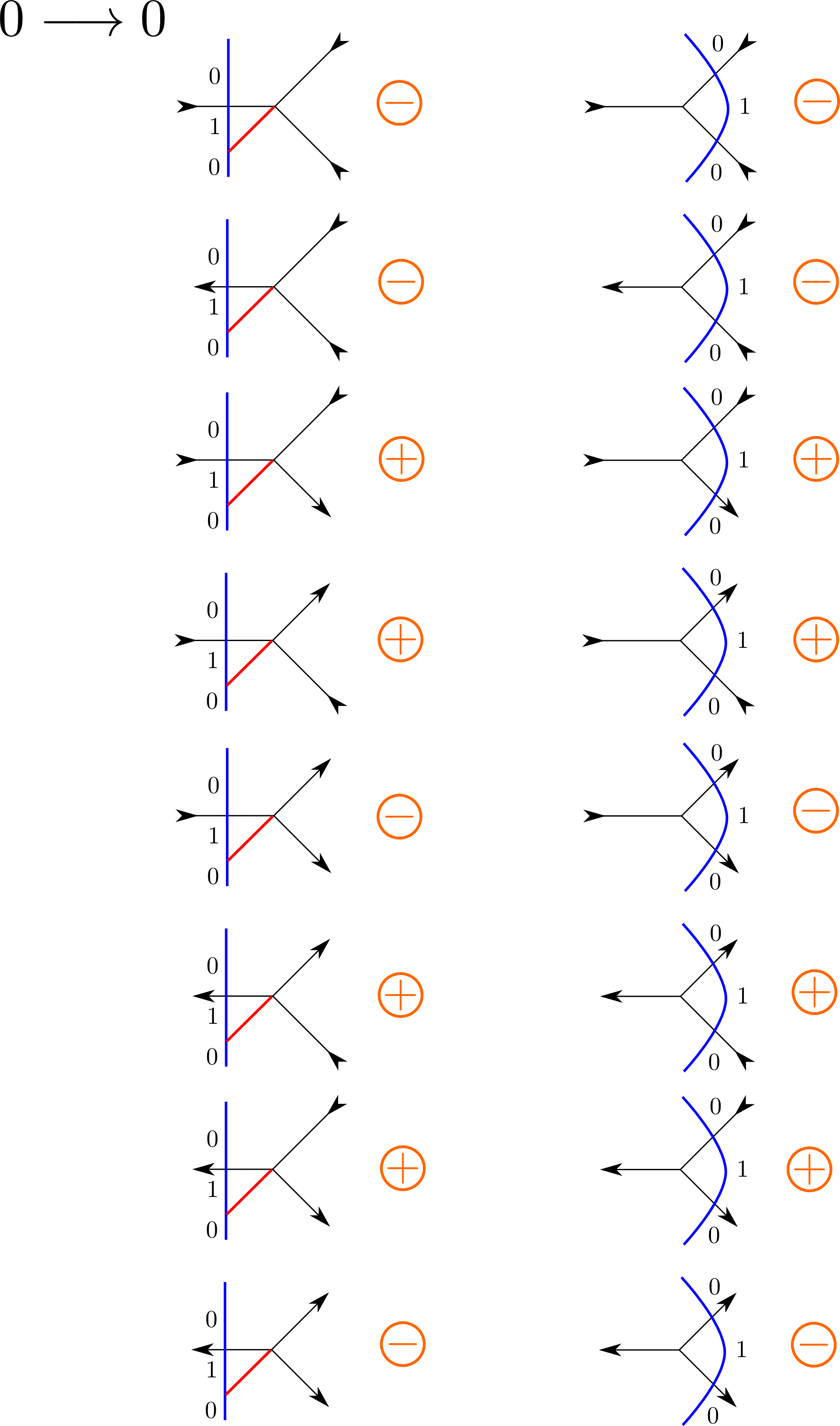}
  \caption{Signs IV}
\end{figure}

\end{document}